\documentclass{amsart}

\usepackage{pb-diagram}
\usepackage{pstricks}

 \newtheorem{thm}{Theorem}[section]
 \newtheorem*{thm*}{Theorem}
 \newtheorem{lem}[thm]{Lemma}
 \newtheorem{prop}[thm]{Proposition}
 \theoremstyle{definition}
 \newtheorem{defn}[thm]{Definition}
 
 \newtheorem{expl}[thm]{Example}
 \theoremstyle{remark}
 \newtheorem*{ack}{Acknowledgements}
 \numberwithin{equation}{section}

\begin{document}

\title[Infinity-Inner-Products on A-Infinity-Algebras]
{Infinity-Inner-Products on A-Infinity-Algebras}

\author[Thomas~Tradler]{Thomas~Tradler}
\address{Thomas Tradler, Department of Mathematics, College of Technology of the City University
of New York, 300 Jay Street, Brooklyn, NY 11201, USA}
\email{ttradler@citytech.cuny.edu}

\begin{abstract}
We give a self contained introduction to A$_\infty$-algebras, A$_\infty$-bimodules and maps between them. The case of A$_\infty$-bimodule-map between $A$ and its dual space $A^{*}$, which we call $\infty$-inner-product, will be investigated in detail. In particular, we describe the graph complex associated to $\infty$-inner-product. In a later paper, we show how $\infty$-inner-products can be used to model the string topology BV-algebra on the free loop space of a Poincar\'e duality space.
\end{abstract}

\maketitle

A$_\infty$-algebras were first introduced by J. Stasheff in \cite{S}
in the study of the homotopy associativity of H-spaces. Since then,
the concept has found numerous applications in many fields of
mathematics and physics. Our interest in A$_\infty$-algebras stems
from applying the concept to the chain level of a Poincar\'e duality
space, and to ultimately obtain a model for string topology
operations defined by M. Chas and D. Sullivan in \cite{CS}. To this
end, it is necessary to develop an appropriate algebraic notion of
Poincar\'e duality. A detailed introduction to such a concept will
be presented in these notes.

In order to describe Poincar\'e duality for a topological space $X$,
note that its cohomology $H$ is an algebra under the cup product,
and thus both $H$ and its dual $H^*$ are bimodules over $H$. With
this, Poincar\'e duality is given by a bimodule equivalence $H(X)\to
H^*(X)$ between the homology and cohomology of $X$. A suitable chain
level version of this may be obtained by considering the bimodule
concept in a homotopy invariant way. Our approach for Poincar\'e
duality consists of examining the A$_\infty$-algebra structure on
the cochains $A$ of $X$. With this, one can see, that both $A$ and
its dual $A^*$ are in fact A$_\infty$-bimodules in an appropriate
sense, described below. Finally, we may complete the analogy by
taking the chain level Poincar\'e duality to be a map between the
A$_\infty$-bimodules $A$ and $A^*$, which we call an
$\infty$-inner-product. We will review these concepts and will
investigate a useful graph complex that is associated to
A$_\infty$-algebras with $\infty$-inner-products. We show how the
graph complex gives rise to a sequence of polyhedra, which include as a special case
Stasheff's associahedra coming from A$_\infty$-algebras.

This paper is the first in a series, setting the foundation for the
algebraic notation necessary to describe Poincar\'e duality at the
chain level, and with this, modeling string topology algebraically.
The next step in this direction is taken in \cite{TZ} in
collaboration with M. Zeinalian, where the following theorem is
proved.
\begin{thm*}[\cite{TZ} Theorem 3.1.4]
Let $X$ be a compact triangulated Poincar\'e duality space, in which
the closure of every simplex is contractible. Then there exists a
symmetric $\infty$-inner-product on the cochains $A$ of $X$, which
on the lowest level induces Poincar\'e duality on homology $H\to H^*$.
\end{thm*}
It is interesting to note, that the lowest level of this theorem
consists of capping with the fundamental cycle of the space, which
does not give a bimodule map at the chain level, but which requires the
notion of A$_\infty$-bimodule maps.

The next step is then to model the string topology operations on the
Hochschild-cochain-complex of an A$_\infty$-algebra $A$ with
$\infty$-inner-product. Let $C^*(A,A)$ and $C^*(A,A^*)$ denote the
Hochschild-cochain-complex of $A$ with values in $A$ and $A^*$,
respectively. The following are some well-known operations on these
space, see e.g. \cite{GJ2}. The $\smile$-product, $\smile:C^{*}(A,A)
\otimes C^{*}(A,A) \to C^{*}(A,A)$, and the Gerstenhaber-bracket,
$[\cdot,\cdot]:C^{*}(A,A) \otimes C^{*}(A,A) \to C^{*}(A,A)$, were
studied in the deformation theory of associative algebras by M.
Gerstenhaber in \cite{G}. On the other hand, if $A$ has a unit $1\in
A$, then we may define Connes' $B$-operator, which may be dualized
to an operation $B:C^{*}(A,A^{*}) \to C^{*}(A,A^{*})$. Using these
operations, it was shown in \cite{T}, that these operations combine
to give a BV-algebra.
\begin{thm*}[\cite{T} Theorem 3.1]
Let $A$ be a unital A$_\infty$-algebra with symmetric and
non-degenerate $\infty$-inner-product. Using the
$\infty$-inner-product, one has an induced quasi-isomorphism of
Hochschild-complexes $C^*(A,A)\to C^*(A,A^*)$. Then, $B$ and
$\smile$ assemble to give a BV-algebra on Hochschild-cohomology,
such that its induced Gerstenhaber-bracket is the one given in
\cite{G}.
\end{thm*}

Combining the theorems from \cite{TZ} and \cite{T}, we obtain a
BV-algebra on the Hochschild-cohomology of the cochains of any Poincar\'e duality
space. This BV-algebra is reminiscent of the BV-algebra from string topology defined on the homology of the free loop space of a compact, oriented manifold $M$, see \cite{CS}.
It is an interesting and non-trivial question, if the identification of the Hochschild-cohomology with the homology of the free loop space also induces an isomorphism of the corresponding BV-algebras. A calculation by L. Menichi in \cite{Me} shows, that in order to recover the BV-algebra on the loop space of $M$, it is in general not enough to consider the cyclic A$_\infty$-algebra on the chain level of $M$, i.e. the A$_\infty$-structure together with the strict Poincar\'e duality inner product on homology. It is our hope, that the notion of $\infty$-inner-products will be strong enough to recover this BV-algebra on the Hochschild-cohomology of $A$, but in any case will help to shed light on this question.

We want to point out, that this BV-algebra is only the tip of the iceberg. In \cite{TZ2} it is shown that there is a whole PROP-action on the Hochschild-cochains of a generalized A$_\infty$-algebra with $\infty$-inner-product. The genus-zero part of this PROP-action constitutes exactly a solution to the cyclic Deligne conjecture for the general case of an A$_\infty$-algebra with $\infty$-inner-product, since it lifts the above BV-algebra on Hochschild-cohomology to the Hochschild-cochain level $C^\ast(A,A^\ast)$. More generally, the full PROP-action in \cite{TZ2} is reminiscent of an action of the moduli-space of Riemann surfaces on the free loop space of a manifold, as it is envisioned by string topology. This shows, that $\infty$-inner-products provide a suitable setup for an algebraic formulation, which captures the ideas from string topology.

The structure of the paper is outlined in the following table of contents.  \tableofcontents

Here are some general remarks on the notation in this paper. Let $R$ denote a commutative ring  with unit. All the spaces ($V$, $W$, $Z$, $A$, $M$, $N$, ...) in this paper are always understood to be graded modules $V=\bigoplus_{i\in \mathbb{Z}}V_{i}$ over $R$, all maps will always be understood as $R$-module maps, and all tensor products will always be assumed over $R$.
The degree of homogeneous elements $v\in V_{i}$ is written as $|v|:=i$, and the
degree of maps $\varphi:V_{i} \to W_{j}$ is written as
$|\varphi|:=j-i$. All tensor-products of maps and their compositions
are understood in a graded way:
$$ (\varphi\otimes\psi)(v\otimes w)= (-1)^{|\psi|\cdot|v|}
   (\varphi(v))\otimes(\psi(w)), $$
$$ (\varphi\otimes\psi)\circ(\chi\otimes\varrho)= (-1)^{|\psi|\cdot|\chi|}
   (\varphi\circ\chi)\otimes(\psi\circ\varrho). $$
All objects $a_{i}$, $v_{i}$,... are assumed to be elements in $A$,
$V$,... respectively, if not stated otherwise.

It will be necessary to look at elements of $V^{\otimes i}\otimes
V^{\otimes j}$. In order to distinguish between the tensor-product
in $V^{\otimes i}$ and the one between $V^{\otimes i}$ and
$V^{\otimes j}$, it is convenient to write the first one as a tuple
$(v_{1},...,v_{i})\in V^{\otimes i}$, and then
$(v_{1},...,v_{i})\otimes (v'_{1},...,v'_{j}) \in V^{\otimes
i}\otimes V^{\otimes j}$. The total degree of $(v_{1},...,v_{i}) \in
V^{\otimes i}$ is given by $ |(v_{1},...,v_{i})|:= \sum_{k=1}^{i}
|v_{k}|$.

Frequently there will be sums of the form $ \sum_{i=0}^{n}
(v_{1},...,v_{i})\otimes(v_{i+1},...,v_{n}).$ Here the convention
will be used that for $i=0$, one has the term $1\otimes
(v_{1},...,v_{n})$ and for $i=n$ the term in the sum is
$(v_{1},...,v_{n})\otimes 1$, with $1=1_{TV}\in TV$. Similarly for
terms $\sum_{i=0}^{k}(v_{1},...,v_{i}) \otimes(v_{i+1},...,v_{k},
w,v_{k+1},...,v_{n})$ the expression for $i=k$ is understood as
$(v_{1},...,v_{k}) \otimes (w,v_{k+1},...,v_{n})$.

\begin{ack}
I would like to thank Dennis Sullivan, who pointed out the
limitations of cyclic A$_\infty$-algebras to me. I am also thankful
to Martin Markl, Jim Stasheff and Mahmoud Zeinalian for many useful comments.
\end{ack}

\section{A$_\infty$-algebras}\label{2}

We first review the definition of A$_\infty$-algebras that are used
in the discussion of this paper.

\begin{defn}\label{2.1}
A \textbf{coalgebra} $(C,\Delta)$ over a ring $R$ consists of an
$R$-module $C$ and a comultiplication  $\Delta:C\to C\otimes C$ of
degree $0$ satisfying coassociativity:
\[
\begin{diagram}
\node{C}\arrow{e,t}{\Delta} \arrow{s,l}{\Delta}
\node{C\otimes C}\arrow{s,r}{\Delta\otimes id}\\
\node{C\otimes C}\arrow{e,b}{id\otimes \Delta} \node{C\otimes
C\otimes C}
\end{diagram}
\]
Then a \textbf{coderivation} on $C$ is a map $f:C\to C$ such that
\[
\begin{diagram}
\node{C}\arrow{e,t}{\Delta} \arrow{s,l}{f}
\node{C\otimes C}\arrow{s,r}{f\otimes id+id\otimes f}\\
\node{C}\arrow{e,b}{\Delta} \node{C\otimes C}
\end{diagram}
\]
\end{defn}

\begin{defn}\label{2.2}
Let $V=\bigoplus_{j\in \mathbb{Z}} V_{j}$ be a graded module over a
given ground ring $R$. The \textbf{tensor-coalgebra} of $V$ over the
ring $R$ is given by
$$ TV:=\bigoplus_{i\geq 0} V^{\otimes i}, $$
$$ \Delta:TV\to TV\otimes TV, \quad
   \Delta(v_{1},...,v_{n}):=\sum_{i=0}^{n} (v_{1},...,v_{i})
    \otimes(v_{i+1},...,v_{n}).$$
Let $A=\bigoplus_{j\in \mathbb{Z}} A_{j}$ be a graded module over the
given ground ring $R$. Define its \textbf{suspension} $sA$ to be the
graded module $sA= \bigoplus_{j\in \mathbb{Z}} (sA)_{j}$ with
$(sA)_{j}:= A_{j-1}$. The suspension map $s:A\to sA$, $s:a\mapsto
sa:=a$ is an isomorphism of degree +1.

Now the \textbf{bar complex} of $A$ is given by $BA:=T(sA)$.

An \textbf{A$_\infty$-algebra} on $A$ is given by a coderivation $D$
on $BA$ of degree $-1$ such that $D^{2}=0$.
\end{defn}

The tensor-coalgebra has the property to lift every module map
$f:TV\to V$ to a coalgebra-map $F:TV \to TV$:
\[
\begin{diagram}
\node{} \node{TV}\arrow{s,r}{projection}\\
\node{TV}\arrow{e,b}{f}\arrow{ne,t}{F} \node{V}
\end{diagram}
\]
A similar property for coderivations on $TV$ will lead to give an
alternative description of A$_\infty$-algebras.
\begin{lem}\label{2.3}
\begin{itemize}
\item [(a)]
Let $\varrho:V^{\otimes n}\to V$, with $n\geq 0$, be a map of degree $|\varrho|$,
which can be viewed as $\varrho:TV\to V$ by letting its only
non-zero component being given by the original $\varrho$ on
$V^{\otimes n}$. Then $\varrho$ lifts uniquely to a coderivation
$\tilde{\varrho} :TV \to TV$ with
\[
\begin{diagram}
\node{} \node{TV}\arrow{s,r}{projection} \\
\node{TV}\arrow{ne,t}{\tilde{\varrho}} \arrow{e,b}{\varrho}
\node{V}
\end{diagram}
\]
by taking
$$ \tilde{\varrho}(v_{1},...,v_{k}):=0, \quad \text{for } k<n, $$
\begin{multline*}
  \quad\quad\quad \tilde{\varrho}(v_{1},...,v_{k}):=\sum_{i=0}^{k-n}
   (-1)^{|\varrho|\cdot(|v_{1}|+...+|v_{i}|)}(v_{1},...,\varrho
   (v_{i+1},...,v_{i+n}),...,v_{k}), \\
\text{for } k\geq n.
\end{multline*}
Thus, $\tilde{\varrho}\mid_{V^{\otimes k}}:V^{\otimes k} \to
V^{\otimes k-n+1} $.
\item [(b)] There is a one-to-one correspondence between
coderivations $\sigma:TV\to TV$ and systems of maps
$\{\varrho_{i}:V^{\otimes i}\to V\}_{i\geq 0}$, given by
$\sigma=\sum_{i\geq 0} \tilde{\varrho_{i}}$.
\end{itemize}
\end{lem}
\begin{proof}
\begin{itemize}
\item [(a)]
Denote by $\tilde{\varrho}^{j}$ the component of $\tilde{\varrho}$
mapping $TV\to V^{\otimes j}$. Then
$\tilde{\varrho}^{1},...,\tilde{\varrho}^ {m-1}$ uniquely determine
the component $\tilde{\varrho}^{m}$, using the coderivation property
of $\tilde{\varrho}$.
\begin{eqnarray*}
\quad\quad
\Delta(\tilde{\varrho}(v_{1},...,v_{k}))&=&(\tilde{\varrho}\otimes
id+id\otimes \tilde{\varrho})(\Delta(v_{1},...,v_{k})) \\
&=& \sum_{i=0}^{k}\tilde{\varrho}(v_{1},...,v_{i})\otimes
(v_{i+1},...,v_{k}) \\
& & \,\,\, +(-1)^{|\tilde{\varrho}|\cdot(|v_{1}|+...+|v_{i}|)}
(v_{1},...,v_{i})\otimes \tilde{\varrho} (v_{i+1},...,v_{k}).
\end{eqnarray*}
Projecting both sides to $V^{\otimes i} \otimes
V^{\otimes j} \subset TV\otimes TV$, with $i+j=m$, yields
\begin{multline*}
\quad\quad\quad \Delta(\tilde{\varrho}^{m}(v_{1},...,v_{k}))|_{V^{\otimes i}\otimes V^{\otimes j}}=
\tilde{\varrho}^{i}(v_{1},...,v_{k-j})\otimes (v_{k-j+1},...,v_{k}) \\
 +(-1)^{|\tilde{\varrho}|\cdot(|v_{1}|+...+|v_{i}|)}
(v_{1},...,v_{i})\otimes \tilde{\varrho}^{j} (v_{i+1},...,v_{k}).
\end{multline*}
For $m=i=1$ and $j=0$, this shows that $\tilde{\varrho}^0=0$. $\tilde{\varrho}^1=\varrho$ by the condition of the Lemma, and for $m\geq 2$, choosing $i=m-1,\,\,\, j=1$ uniquely determines $\tilde{\varrho}^m$ by lower components. Thus, an induction shows, that $\tilde{\varrho}^{m}$ is only nonzero on
$V^{\otimes k}$ for $k=m+n-1$, where $\tilde{\varrho}^{m}
(v_{1},...,v_{m+n-1})$ is given by $ \sum_{i=0}^{m-1}
   (-1)^{|\varrho|\cdot(|v_{1}|+...+|v_{i}|)}  (v_{1},..., \varrho
   (v_{i+1},...,v_{i+n}),...,v_{m+n-1})$.

\item [(b)]
The map
$$ \quad\quad\quad
\alpha:\{\{\varrho_{i}:V^{\otimes i}\rightarrow V\}_{i\geq 0}\} \to
Coder(TV), \quad \{\varrho_{i}: V^{\otimes i}\rightarrow V\}_{i\geq
0} \mapsto \sum_{i\geq 0} \tilde{\varrho_{i}} $$ is well defined.
Its inverse $\beta$ is given by
$\beta:\sigma\mapsto\{pr_{V}\circ\sigma|_{V^{\otimes i }}\}_{i\geq
0}$, because the explicit lifting property of (a) shows that
$\beta\circ\alpha=id$, and the uniqueness part of (a) shows that
$\alpha\circ\beta=id$.
\end{itemize}
\end{proof}

Application to Definition \ref{2.2} gives the following
\begin{prop}\label{2.4}
Let $(A,D)$ be an A$_\infty$-algebra, and let $D$ be given by a
system of maps $\{D_{i}:sA^{\otimes i} \to sA\}_{i\geq 1}$, with
$D_{0}=0$. Let $m_{i}:A^{\otimes i} \to A$ be given by $D_{i}=s\circ
m_{i}\circ (s^{-1})^{\otimes i}$. Then the condition $D^{2}=0$ is
equivalent to the following system of equations:
\begin{eqnarray*}
m_{1}(m_{1}(a_{1}))&=&0,\\
m_{1}(m_{2}(a_{1},a_{2}))-m_{2}(m_{1}(a_{1}),a_{2})-(-1)^{|a_{1}|}
m_{2}(a_{1},m_{1}(a_{2}))&=&0,\\
m_{1}(m_{3}(a_{1},a_{2},a_{3}))-m_{2}(m_{2}(a_{1},a_{2}),a_{3})+
m_{2}(a_{1},m_{2}(a_{2},a_{3}))& &\\
+m_{3}(m_{1}(a_{1}),a_{2},a_{3})+(-1)^{|a_{1}|}m_{3}
(a_{1},m_{1}(a_{2}),a_{3})& &\\
+(-1)^{|a_{1}|+|a_{2}|}m_{3}(a_{1},a_{2},m_{1}(a_{3}))&=&0,\\
...\\
\sum_{i=1}^{k} \sum_{j=0}^{k-i+1} (-1)^{\varepsilon} \cdot
m_{k-i+1} (a_{1},...,m_{i} (a_{j},...,a_{j+i-1}),...,a_{k})&=&0,\\
where \,\,\,\varepsilon=i\cdot \sum_{l=1}^{j-1}|a_{l}|+ (j-1)
\cdot(i+1)+k-i\\
...
\end{eqnarray*}
\end{prop}
\begin{proof}
This follows from Lemma \ref{2.3} after a careful check of the
involved signs.
\end{proof}

\begin{expl}\label{2.5}
Any differential graded algebra $(A,\partial,\mu)$ gives an
A$_\infty$-algebra-structure on $A$ by taking $m_{1}:=\partial$,
$m_{2}:=\mu$ and $m_{k}:=0$ for $k\geq 3$. The equations from
Proposition \ref{2.4} are the defining conditions of a differential
graded algebra:
\begin{eqnarray*}
\partial^{2}(a)&=&0,\\
\partial (a\cdot b)&=& \partial(a)\cdot b+ (-1)^{|a|}a\cdot \partial(b), \\
(a\cdot b)\cdot c &=& a\cdot (b\cdot c). \\
\end{eqnarray*}
There are no higher equations.
\end{expl}

\begin{defn}\label{2.6}
Let $(A,D)$ be an A$_\infty$-algebra. The
\textbf{Hochschild-cochain-complex of $A$} is defined to be the
space $C^{*}(A):=CoDer(BA,BA)$ of coderivations on $BA$ with the
differential $\delta:C^{*}(A)\to C^{*}(A)$ given by
$\delta(f):=[D,f]=D\circ f-(-1)^{|f|}f\circ D$. We have
$\delta^{2}=0$, because with $D$ of degree $-1$ and $D^{2}=0$, it
follows that $\delta^{2}(f)=[D,D\circ f-(-1)^{|f|}f\circ D]= D\circ
D\circ f -(-1)^{|f|} D\circ f\circ D -(-1)^{|f|+1} D\circ f \circ D
+ (-1)^{|f|+|f|+1} f\circ D \circ D=0$.
\end{defn}

\section{A$_\infty$-bimodules}\label{3}

Let $(A,D)$ be an A$_\infty$-algebra. We now define the concept of
an A$_\infty$-bimodule over $A$, which was also considered in
\cite{GJ} and \cite{M}. This should be a generalization of two
facts. First, it is possible to define the
Hochschild-cochain-complex for any algebra with values in a
bimodule, which we would also like to do in the A$_\infty$ case.
Second, any algebra is a bimodule over itself by left- and
right-multiplication, which should also hold in the A$_\infty$ case.
The following space and map are important ingredients.
\begin{defn}\label{3.1}
For modules $V$ and $W$ over $R$, we define
$$T^{W}V:=R\oplus\bigoplus_{k\geq 0, l\geq 0} V^{\otimes k} \otimes W
\otimes V^{\otimes l}.$$
Furthermore, let
\begin{eqnarray*}
\Delta^{W}:T^{W}V&\to & (TV\otimes T^{W}V)\oplus (T^{W}V\otimes TV),\\
\Delta^{W}(v_{1},...,v_{k},w,v_{k+1},...,v_{k+l})&:=&
 \sum_{i=0}^{k}(v_{1},...,v_{i}) \otimes(v_{i+1},...,w,...,v_{n})\\
 & &+\sum_{i=k}^{k+l} (v_{1},...,w,...,v_{i})
 \otimes(v_{i+1},...,v_{k+l}).
\end{eqnarray*}
Again for modules $A$ and $M$ let $B^{M}A$ be given by $T^{sM}sA$,
where $s$ is the suspension from Definition \ref{2.2}.
\end{defn}\label{3.2}
Observe that $T^{W}V$ is not a coalgebra, but rather a bi-comodule
over $TV$. We need the definition of a coderivation from $TV$ to
$T^{W}V$.
\begin{defn} A \textbf{coderivation} from $TV$ to $T^{W}V$ is a map
$f:TV\to T^{W}V$ so that the following diagram commutes:
\[
\begin{diagram}
\node{TV}\arrow{s,l}{f}\arrow{e,t}{\Delta}
  \node{TV\otimes TV}\arrow{s,r}{id\otimes f+f\otimes id} \\
\node{T^{W}V}\arrow{e,b}{\Delta^{W}} \node{(TV\otimes
T^{W}V)\oplus (T^{W}V\otimes TV)}
\end{diagram}
\]
For modules $A$ and $M$ let $C^{*}(A,M):=CoDer(BA,B^{M}A)$ be the
space of coderivations in the above sense, called the
\textbf{Hochschild-cochain-complex of $A$ with values in $M$}.
\end{defn}
\begin{lem}\label{3.3}
\begin{itemize}
\item [(a)]
Let $\varrho:V^{\otimes n}\to W$ be a map of degree $|\varrho|$,
which can be viewed as a map $\varrho:TV\to W$ by letting its only
non-zero component being given by the original $\varrho$ on
$V^{\otimes n}$. Then $\varrho$ lifts uniquely to a coderivation
$\tilde{\varrho} :TV \to T^{W}V$ with
\[
\begin{diagram}
\node{} \node{T^{W}V}\arrow{s,r}{projection} \\
\node{TV}\arrow{ne,t}{\tilde{\varrho}} \arrow{e,b}{\varrho}
\node{W}
\end{diagram}
\]
by taking
$$ \tilde{\varrho}(v_{1},...,v_{k}):=0, \quad \text{for } k<n, $$
\begin{multline*}
 \quad\quad\quad\tilde{\varrho}(v_{1},...,v_{k}):=\sum_{i=0}^{k-n}
   (-1)^{|\varrho|\cdot(|v_{1}|+...+|v_{i}|)}
   (v_{1},...,\varrho(v_{i+1},...,v_{i+n}),...,v_{k}),\\
 \text{for } k\geq n.
\end{multline*}
Thus $\tilde{\varrho}\mid_{V^{\otimes k}}:V^{\otimes k} \to
\bigoplus_{i+j=k-n} V^{\otimes i} \otimes W \otimes V^{\otimes j}$.
\item [(b)] There is a one-to-one correspondence between
coderivations $\sigma:TV\to T^{W}V$ and systems of maps
$\{\varrho_{i}:V^{\otimes i}\to W\}_{i\geq 0}$, given by
$\sigma=\sum_{i\geq 0} \tilde{\varrho_{i}}$.
\end{itemize}
\end{lem}
\begin{proof}
\begin{itemize}
\item [(a)]
The proof is similar to the one of Lemma \ref{2.3} (a). Let
$\tilde{\varrho}^{j}$ be the component of $\tilde{\varrho}$ mapping
$TV\to\bigoplus_{r+s=j} V^{\otimes r} \otimes W\otimes V^{\otimes
s}$, and $\tilde{\varrho}^{-1}$ the component $TV\to R$. The equation
\begin{eqnarray*}
\quad\quad\quad
\Delta^{W}(\tilde{\varrho}(v_{1},...,v_{k}))&=&(\tilde{\varrho}\otimes
id+id\otimes \tilde{\varrho})(\Delta(v_{1},...,v_{k})) \\
&=& \sum_{i=0}^{k}\tilde{\varrho}(v_{1},...,v_{i})\otimes
(v_{i+1},...,v_{k}) \\
& & \,\,\, +(-1)^{|\tilde{\varrho}|\cdot(|v_{1}|+...+|v_{i}|)}
(v_{1},...,v_{i})\otimes \tilde{\varrho} (v_{i+1},...,v_{k})
\end{eqnarray*}
projected to $R\otimes TV$ shows that $\tilde{\varrho}^{-1}=0$. $\tilde{\varrho}^{0}=\varrho$ is uniquely determined by the statement of the Lemma, and projecting for fixed $i+j=m$, to the component
\begin{multline*}
\quad\quad\quad \bigoplus_{r+s=i} (V^{\otimes r}\otimes W\otimes
V^{\otimes s})\otimes V^{\otimes j}+V^{\otimes j}\otimes \bigoplus_{r+s=i} (V^{\otimes r}\otimes W\otimes V^{\otimes s})\\
\subset T^{W}V\otimes TV+TV\otimes T^{W}V,
\end{multline*}
shows that $ \Delta^{W}(\tilde{\varrho}^{m}(v_{1},...,v_{k}))|_{\bigoplus_{r+s=i} (V^{r}\otimes W\otimes V^{s})\otimes V^{j}+V^{j}\otimes \bigoplus_{r+s=i} (V^{r}\otimes W\otimes V^{s})}$ is given by 
\begin{equation*}
\quad\quad  \tilde{\varrho}^{i}(v_{1},...,v_{k-j})\otimes (v_{k-j+1},...,v_{k})+(-1)^{|\tilde{\varrho}|\cdot(|v_{1}|+...+|v_{j}|)} (v_{1},...,v_{j})\otimes \tilde{\varrho}^i (v_{j+1},...,v_{k}) .
\end{equation*}
For $m\geq 1$, choosing $i=m-1,\,\,\, j=1$ uniquely determines $\tilde{\varrho}^{m}$ by lower components. Thus, an induction shows, that $\tilde{\varrho}^{m}$ is only nonzero on
$V^{\otimes k}$ for $k=m+n-1$, where $ \tilde{\varrho}^{m}
(v_{1},...,v_{m+n-1})$ is given by $\sum_{i=0}^{m-1}
   (-1)^{|\varrho|\cdot(|v_{1}|+...+|v_{i}|)}\cdot \\ (v_{1},...,\varrho
   (v_{i+1},...,v_{i+n}),...,v_{m+n-1})$.
\item [(b)]
Then maps
\begin{eqnarray*}
& \alpha:\{\{\varrho_{i}:V^{\otimes i}\rightarrow W\}_{i\geq 0}\}
\to Coder(TV,T^{W}V), & \{\varrho_{i}: V^{\otimes i}\rightarrow
W\}_{i\geq 0} \mapsto
\sum_{i\geq 0} \tilde{\varrho_{i}} \\
& \beta:Coder(TV,T^{W}V) \to \{\{\varrho_{i}: V^{\otimes
i}\rightarrow W\}_{i\geq 0}\}, &
\sigma\mapsto\{pr_{W}\circ\sigma|_{V^{\otimes i }}\}_{i\geq 0}
\end{eqnarray*}
are inverse to each other by (a).
\end{itemize}
\end{proof}
We put a differential on $C^{*}(A,M)$, similar to the one from
section \ref{2}.
\begin{prop}\label{3.4}
Let $(A,D)$ be an A$_{\infty}$-algebra and $M$ be a graded
module. Let $D^{M}:B^{M}A\to B^{M}A$ be a map of degree $-1$. Then
the induced map $\delta^{M}:CoDer(BA,B^{M}A) \to CoDer(BA,B^{M}A)$,
given by $\delta^{M}(f):=D^{M}\circ f-(-1)^{|f|}f\circ D$, is
well-defined, (i.e. it maps coderivations to coderivations,) if and
only if the following diagram commutes:
\begin{equation}\label{eqn-3.1}
\begin{diagram}
\node{B^{M}A}\arrow{s,l}{D^{M}}\arrow{e,t}{\Delta^{M}}
  \node{(BA\otimes B^{M}A)\oplus (B^{M}A\otimes BA)}\arrow{s,r}
  {(id\otimes D^{M}+D\otimes id) \oplus (D^{M}\otimes id+id\otimes D)} \\
\node{B^{M}A}\arrow{e,b}{\Delta^{M}} \node{(BA\otimes
  B^{M}A)\oplus (B^{M}A\otimes BA)}
\end{diagram}
\end{equation}
\end{prop}
\begin{proof}
Let $f:BA\to B^{M}A$ be a coderivation. Then, $\delta^{M}(f)$ is a
coderivation, if
\begin{multline*}
 (id\otimes\delta^{M}(f)+\delta^{M}(f)\otimes id)\circ\Delta=
\Delta^{M}\circ\delta^{M}(f) \text{, i.e.}\\
 (id\otimes(D^{M}\circ f)-(-1)^{|f|}id\otimes(f\circ D)+
    (D^{M}\circ f)\otimes id-(-1)^{|f|}
    (f\circ D)\otimes id)\circ\Delta \\= \Delta^{M}\circ D^{M}\circ f-(-1)^{|f|}
    \Delta^{M}\circ f \circ D.
\end{multline*}
Using the coderivation property for $f$ and $D$, we get
\begin{eqnarray*}
\Delta^{M}\circ f \circ D &=&
 (id\otimes f)\circ\Delta\circ D + (f\otimes id)\circ\Delta\circ D\\
 &=& (id\otimes(f\circ D) +(-1)^{|f|} D\otimes f + f\otimes D +
      (f\circ D)\otimes id) \circ \Delta,
\end{eqnarray*}
so that the requirement for $\delta^{M}(f)$ being a coderivation
reduces to
\begin{eqnarray*}
\Delta^{M}\circ D^{M}\circ f&=& (id\otimes(D^{M}\circ f)+
(D^{M}\circ f)\otimes id+ D\otimes f +(-1)^{|f|} f\otimes D)\circ \Delta\\
&=& (id\otimes D^{M}+ D\otimes id)\circ(id\otimes f)\circ \Delta\\
&& +
    (D^{M}\otimes id+ id\otimes D)\circ(f\otimes id)\circ \Delta\\
&=& (id\otimes D^{M}+ D\otimes id)\circ\Delta^{M} \circ f +
    (D^{M}\otimes id+ id\otimes D)\circ\Delta^{M} \circ f.
\end{eqnarray*}
Thus, we get the following condition for $D^{M}$,
$$ \Delta^{M}\circ D^{M}\circ f = (id\otimes D^{M}+ D\otimes id+
    D^{M}\otimes id+ id\otimes D)\circ\Delta^{M} \circ f $$
for all coderivations $f:TA\to T^{M}A$. With Lemma \ref{3.3} this
condition reduces to $ \Delta^{M}\circ D^{M}= (id\otimes D^{M}+
D\otimes id+ D^{M}\otimes id+ id\otimes D)\circ\Delta^{M}$, which is
the claim.
\end{proof}
We can describe $D^{M}$ by a system of maps.
\begin{lem}\label{3.5}
\begin{itemize}
\item [(a)]
Let $V$ be a module, and let $\psi$ be a coderivation on $TV$
 with associated system of maps $\{\psi_{i}:V^{\otimes i} \to
 V\}_{i\geq 1}$ from Lemma \ref{2.3}. Then any map $\varrho:T^{W}V\to W$
 given by $\varrho=\sum_{k\geq0, l\geq0} \varrho_{k,l}$, with
 $\varrho_{k,l}: V^{\otimes k}\otimes W \otimes V^{\otimes l}
 \to W$, lifts uniquely to a map
 $\tilde{\varrho} :T^{W}V \to T^{W}V$
\[
\begin{diagram}
\node{} \node{T^{W}V}\arrow{s,r}{projection} \\
\node{T^{W}V}\arrow{e,b}{\varrho} \arrow{ne,t}{\tilde{\varrho}}
\node{W}
\end{diagram}
\]
which makes the following diagram commute
\begin{equation}\label{eqn-3.2}
\begin{diagram}
\node{T^{W}V}\arrow{s,l}{\tilde{\varrho}}\arrow{e,t}{\Delta^{W}}
  \node{(TV\otimes T^{W}V)\oplus (T^{W}V\otimes TV)}\arrow{s,r}
  {(id\otimes \tilde{\varrho}+\psi\otimes id) \oplus (\tilde{\varrho}
  \otimes id+id\otimes \psi)} \\
\node{T^{W}V}\arrow{e,b}{\Delta^{W}} \node{(TV\otimes
  T^{W}V)\oplus (T^{W}V\otimes TV)}
\end{diagram}
\end{equation}
This map is given
\begin{multline*}
 \tilde{\varrho}(v_{1},...,v_{k},w,v_{k+1},...,v_{k+l}) \\
:= \sum_{i=1}^{k} \sum_{j=1}^{k-i+1}
     (-1)^{|\psi_{i}|\sum_{r=1}^{j-1}|v_{r}|}
     (v_{1},...,\psi_{i}(v_{j},...,v_{i+j-1}),...,w,...,v_{k+l}) \quad\quad\quad\quad\quad \\
   +\sum_{i=0}^{k} \sum_{j=0}^{l}
     (-1)^{|\varrho_{i,j}|\sum_{r=1}^{k-i}|v_{r}|}
     (v_{1},...,\varrho_{i,j}(v_{k-i+1},...,w,...,v_{k+j}),...,v_{k+l}) \quad\quad\quad\quad \\
   +\sum_{i=1}^{l} \sum_{j=1}^{l-i+1}
     (-1)^{|\psi_{i}|(|w|+\sum_{r=1}^{k+j-1}|v_{r}|)}
     (v_{1},...,w,...,\psi_{i}(v_{k+j},...,v_{k+i+j-1}),...,v_{k+l}).
\end{multline*}
(Notice that the condition of diagram \eqref{eqn-3.2} is not
linear.)
\item [(b)]
There is a one-to-one correspondence between maps $\sigma:T^{W}V\to
T^{W}V$ that make diagram \eqref{eqn-3.2} commute and maps
$\varrho=\sum\varrho_{k,l}$ from (a), given by
$\sigma=\tilde{\varrho}$.
\end{itemize}
\end{lem}
\begin{proof}
\begin{itemize}
\item [(a)]
As in the Lemmas \ref{2.3} and \ref{3.3}, we denote by $\tilde{\varrho}^{j}$, $j \geq 0$, the component of $\tilde{\varrho}$ mapping $T^{W}V\to \bigoplus _{k+l=j} V^{\otimes k}\otimes W\otimes V^{\otimes l}$ and by $\tilde{\varrho}^{-1}$ the component $T^W V\to R$. $\psi^{j}$, for $j\geq 1$, denotes the component of $\psi$ mapping
$TV\to V^{\otimes j}$. Then $\tilde {\varrho}^{m}$ is uniquely
determined by $\tilde{\varrho}^{0},...,\tilde{\varrho}^ {m-1}$.
\begin{multline*}
\quad\quad\quad \Delta^{W}(\tilde{\varrho}
(v_{1},...,v_{k},w,v_{k+1},...,v_{k+l})) \\
=(id\otimes \tilde{\varrho}+\psi\otimes id+ \tilde{\varrho}
   \otimes id+id\otimes\psi) (\Delta^{W}
   (v_{1},...,v_{k},w,v_{k+1},...,v_{k+l}))
\end{multline*}
\begin{eqnarray*}
&=& \sum_{i=0}^{k}(-1)^{|\tilde{\varrho}|\sum_{r=1}^{i}|v_{r}|}
    (v_{1},...,v_{i}) \otimes \tilde{\varrho}(v_{i+1},...,w,...,v_{k+l})\\
& & +\sum_{i=0}^{k}\psi(v_{1},...,v_{i}) \otimes
(v_{i+1},...,w,...,v_{k+l})\\
& &  +\sum_{i=k}^{k+l}\tilde{\varrho}(v_{1},...,w,...,v_{i})
    \otimes (v_{i+1},...,v_{k+l})\\
& & +\sum_{i=k}^{k+l} (-1)^{|\psi|(|w|+\sum_{r=1}^{i}|v_{r}|)}
    (v_{1},...,w,...,v_{i}) \otimes \psi(v_{i+1},...,v_{k+l}).\,\,
\end{eqnarray*}
Projecting both sides to $R\otimes TV$ shows that $\tilde{\varrho}^{-1}=0$, and projecting for fixed $i+j=m$, to the component
\begin{multline*}
\quad\quad\quad V^{\otimes j}\otimes \bigoplus_{r+s=i} (V^{\otimes r}\otimes W\otimes V^{\otimes s})+ \bigoplus_{r+s=i} (V^{\otimes r}\otimes W\otimes V^{\otimes s})\otimes V^{\otimes j} \\
\subset T^{W}V\otimes TV+TV\otimes T^{W}V,
\end{multline*}
shows that $ \Delta^{W}(\tilde{\varrho}^{m}(v_{1},...,w,...,v_{k}))|_{V^{\otimes j}\otimes \bigoplus_{r+s=i} (V^{\otimes r}\otimes W\otimes V^{\otimes s})+ \bigoplus_{r+s=i} (V^{\otimes r}\otimes W\otimes V^{\otimes s})\otimes V^{\otimes j} }$ is given by 
\begin{multline*}
\pm (v_{1},...,v_{j})\otimes \tilde{\varrho}^i (v_{j+1},...,w,...,v_{k+l}) + \psi^{j}(v_1,...,v_{k+l-i})\otimes (v_{k+l-i+1},...,w,...,v_{k+l}) \\
+ \tilde{\varrho}^{i}(v_{1},...,w,...,v_{k+l-j})\otimes (v_{k+l-j+1},...,v_{k+l})\pm (v_1,...,w,...v_i)\otimes \psi^j (v_{i+1},...,v_{k+l}).
\end{multline*}
For $m\geq 1$, choosing $i=m-1,\,\,\, j=1$ uniquely determines $\tilde{\varrho}^{m}$ by lower components, and the $\psi^j$'s. Thus, an induction shows the claim of the Lemma.
\item [(b)]
Let $X:=\{\sigma:T^{W}V\to T^{W}V\,\,|\,\, \sigma$ makes diagram
\eqref{eqn-3.2} commute$\}$. Then
\begin{eqnarray*}
\alpha:\{\varrho:T^{W}V\to W\}\to X,
& & \varrho \mapsto \tilde{\varrho}, \\
\beta:X\to \{\varrho:T^{W}V\to W\}, & & \sigma \mapsto
pr_{W}\circ\sigma
\end{eqnarray*}
are inverse to each other by (a).
\end{itemize}
\end{proof}
\begin{defn}\label{3.6}
Let $(A,D)$ be an A$_{\infty}$-algebra. Then an
\textbf{A$_{\infty}$-bimodule} $(M,D^{M})$ consists of a graded
module $M$ together with a map $D^{M}:B^{M}A\to B^{M}A$ of degree
$-1$, which makes the diagram \eqref{eqn-3.1} of Proposition
\ref{3.4} commute, and satisfies $(D^{M})^{2} =0$.

By Proposition \ref{3.4}, we may put the differential $\delta^{M}:
CoDer(TA,T^{M}A) \to CoDer(TA,T^{M}A)$, $\delta(f) :=D^{M}\circ
f-(-1)^{|f|}f\circ D$ on the Hochschild-cochain-complex. It
satisfies $(\delta^{M})^{2}=0$, because with $(D^{M})^{2} =0$, we
get $(\delta^{M})^{2}(f)=D^{M}\circ D^{M}\circ
f-(-1)^{|f|}D^{M}\circ f\circ D -(-1)^{|f|+1} D^{M}\circ f \circ D +
(-1)^{|f|+|f|+1} f\circ D\circ D =0$.

The definition of an A$_{\infty}$-bimodule was already stated in
\cite{GJ} section 3 and also in \cite{M}.
\end{defn}
\begin{prop}\label{3.7}
Let $(A,D)$ be an A$_\infty$-algebra, and let $\{m_{i}:A^{\otimes i}
\to A\}_{i\geq 1}$ be the system of maps associated to $D$ by
Proposition \ref{2.4}, with $m_{0}=0$. Let $(M,D^{M})$ be an
A$_\infty$-bimodule over $A$, and let $\{D^{M}_{k,l} :sA^{\otimes
k}\otimes sM \otimes sA^{\otimes l} \to sM\}_{k\geq 0, l\geq 0}$ be
the system of maps associated to $D^{M}$ by Lemma \ref{3.5} (b). Let
$b_{k,l}:A^{\otimes k}\otimes M \otimes A^{\otimes l} \to M$ be the
induced maps by $D^{M}_{k,l}=s\circ b_{k,l}\circ (s^{-1})^{\otimes
k+l+1}$. Then the condition $(D^{M})^{2}=0$ is equivalent to the
following system of equations:
\begin{eqnarray*}
b_{0,0}(b_{0,0}(m))&=&0, \\
b_{0,0}(b_{0,1}(m,a_{1}))-b_{0,1}(b_{0,0}(m),a_{1})-(-1)^{|m|}
b_{0,1}(m,m_{1}(a_{1})) &=& 0,\\
b_{0,0}(b_{1,0}(a_{1},m))-b_{1,0}(m_{1}(a_{1}),m)-(-1)^{|a_{1}|}
b_{1,0}(a_{1},b_{0,0}(m)) &=& 0,\\
b_{0,0}(b_{1,1}(a_{1},m,a_{2}))-b_{0,1}(b_{1,0}(a_{1},m),a_{2})+
b_{1,0}(a_{1},b_{0,1}(m,a_{2}))& &\\
+b_{1,1}(m_{1}(a_{1}),m,a_{2})
+(-1)^{|a_{1}|} b_{1,1}(a_{1},b_{0,0}(m),a_{2})& &\\
+(-1)^{|a_{1}|+|m|}b_{1,1}(a_{1},m,m_{1}(a_{2})) &=& 0,\\
\end{eqnarray*}
$$...$$
\begin{eqnarray*}
\sum_{i=1}^{k} \sum_{j=1}^{k-i+1} \pm
 b_{k-i+1,l}(a_{1},...,m_{i}(a_{j},...,a_{i+j-1}),...,m,...,a_{k+l}) & &\\
+\sum_{i=0}^{k} \sum_{j=0}^{l} \pm
 b_{k-i,l-j}(a_{1},...,b_{i,j}(a_{k-i+1},...,m,...,a_{k+j}),...,a_{k+l}) & &\\
+\sum_{i=1}^{l} \sum_{j=1}^{l-i+1} \pm
 b_{k,l-i+1}(a_{1},...,m,...,m_{i}(a_{k+j},...,a_{k+i+j-1}),...,a_{k+l})&=&0
\end{eqnarray*}
$$...$$
where the signs are analogous to the ones in Proposition \ref{2.4}.
\end{prop}
\begin{proof}
The result follows from Lemma \ref{3.5}, after rewriting
$D^{M}_{k,l}$ and $D_{j}$ by $b_{k,l}$ and $m_{j}$.
\end{proof}
\begin{expl}\label{3.8}
With this, Example \ref{2.5} may be extended in the following way.
Let $(A,\partial,\mu)$ be a differential graded algebra with the
A$_\infty$-algebra-structure $m_{1}:=\partial$, $m_{2}:=\mu$ and
$m_{k}:=0$ for $k\geq 3$. Let $(M,\partial',\lambda,\rho)$ be a
differential graded bimodule over $A$, where $\lambda:A\otimes M\to
M$ and $\rho :M\otimes A\to M$ denote the left- and right-action,
respectively. Then, $M$ is an A$_\infty$-bimodule over $A$ by taking
$b_{0,0}:=\partial'$, $b_{1,0}:=\lambda$, $b_{0,1}:=\rho$ and
$b_{k,l}:=0$ for $k+l>1$. The equations of Proposition \ref{3.7} are
the defining conditions for a differential bialgebra over $A$:
\begin{eqnarray*}
(\partial')^{2}(m)&=&0,\\
\partial'(m.a)&=&m.\partial(a)+(-1)^{|m|}\partial'(m).a,\\
\partial'(a.m)&=&\partial(a).m+(-1)^{|a|}a.\partial'(m),\\
(a.m).b&=&a.(m.b),\\
(m.a).b&=&m.(a\cdot b),\\
a.(b.m)&=&(a\cdot b).m.
\end{eqnarray*}
There are no higher equations.
\end{expl}

For later purposes it is convenient to have the following
\begin{lem}\label{3.9}
Given an A$_\infty$-algebra $(A,D)$ and an A$_\infty$-bimodule
$(M,D^{M})$, with system of maps $\{b_{k,l} :A^{\otimes k}\otimes M
\otimes A^{\otimes l} \to M\}_{k\geq 0, l\geq 0}$ from Proposition
\ref{3.7}, then the dual space $M^{*}:=Hom_{R}(M,R)$ has a canonical
A$_\infty$-bimodule-structure given by maps $\{b'_{k,l} :A^{\otimes
k}\otimes M^{*} \otimes A^{\otimes l} \to M^{*}\}_{k\geq 0, l\geq
0}$,
\begin{multline*}
(b'_{k,l}(a_{1},...,a_{k},m^{*},a_{k+1},...,a_{k+l}))(m):=(-1)^
 {\varepsilon}\cdot m^{*}(b_{l,k}(a_{k+1},...,a_{k+l},m,a_{1},...,a_{k})),\\
\text{ where } \varepsilon:= (|a_{1}|+...+|a_{k}|)\cdot
(|m^{*}|+|a_{k+1}|+...+|a_{k+l}|+|m|)+|m^{*}|\cdot(k+l+1).
\end{multline*}
\end{lem}
\begin{proof}
To see, that $(D^{M^{*}})^{2}=0$, we can use the criterion from
Proposition \ref{3.7}. The top and the bottom term in the general
sum of Proposition \ref{3.7} convert to
\begin{multline*}
(b'_{k-i+1,l}(a_{1},...,m_{i}(a_{j},...,a_{i+j-1})
    ,...,m^{*},...,a_{k+l}))(m)\\
    =\pm m^{*}(b_{l,k-i+1}(a_{k+1},...,a_{k+l},m,
a_{1},...,m_{i}(a_{j},...,a_{i+j-1}),...,a_{k}))\text{, and}
\end{multline*}
\begin{multline*}
(b'_{k,l-i+1}(a_{1},...,m^{*},...,m_{i}(a_{k+j},...,
   a_{k+i+j-1}),...,a_{k+l}))(m)\\
   =\pm m^{*}(b_{l-i+1,k}(a_{k+1},...,m_{i}(a_{k+j},...,
   a_{k+i+j-1}),...,a_{k+l},m,a_{1},...,a_{k})).
\end{multline*}
These terms come from the A$_\infty$-bimodule-structure of $M$.
Similar arguments apply to the middle term:
\begin{multline*}
(b'_{k-i,l-j}(a_{1},...,b'_{i,j}(a_{k-i+1},...,
 m^{*},...,a_{k+j}),...,a_{k+l}))(m)\\
= \pm (b'_{i,j}(a_{k-i+1},...,m^{*},...,a_{k+j}))(b_{l-j,k-i}
      (a_{k+j+1},...,a_{k+l},m,a_{1},...,a_{k-i}))\quad\quad\quad\,\, \\
= \pm m^{*}(b_{j,i}(a_{k+1},...,a_{k+j},b_{l-j,k-i}
      (a_{k+j+1},...,a_{k+l},m,a_{1},...,a_{k-i}),a_{k-i+1},...,a_{k})).
\end{multline*}
The sum from Proposition \ref{3.7} for the A$_\infty$-bimodule
$M^{*}$ contains exactly the terms of $m^{*}$ applied the the sum
for the A$_\infty$-bimodule $M$. A thorough check identifies the
signs.
\end{proof}

\section{Morphisms of A$_\infty$-bimodules}\label{4}

Let $(M,D^{M})$ and $(N,D^{N})$ be two A$_\infty$-bimodules over the
A$_\infty$-algebra $(A,D)$. We next define the notion of
A$_\infty$-bimodule-map between $(M,D^{M})$ and $(N,D^{N})$. Again a
motivation is to have an induced map of their
Hochschild-cochain-complexes.

\begin{prop}\label{4.1}
Let $V$, $W$ and $Z$ be modules, and let $F$ be a map $F:T^{W}V \to
T^{Z}V$. Then the induced map $F^{\sharp} :CoDer(TV,T^{W}V) \to
CoDer(TV,T^{Z}V)$, given by $F^{\sharp}(f):=F\circ f$, is
well-defined, (i.e. it maps coderivations to coderivations,) if and
only if the following diagram commutes:
\begin{equation}\label{eqn-4.1}
\begin{diagram}
\node{T^{W}V}\arrow{s,l}{F}\arrow{e,t}{\Delta^{W}}
  \node{(TV\otimes T^{W}V)\oplus (T^{W}V\otimes TV)}\arrow{s,r}
  {(id\otimes F) \oplus (F\otimes id)} \\
\node{T^{Z}V}\arrow{e,b}{\Delta^{Z}} \node{(TV\otimes
  T^{Z}V)\oplus (T^{Z}V\otimes TV)}
\end{diagram}
\end{equation}
\end{prop}
\begin{proof}
If both $f:TV\to T^{W}V$ and $F\circ f:TV\to T^{Z}V$ are
coderivations, then the top diagram and the overall diagram below
commute.
\[
\begin{diagram}
\node{TV}\arrow{s,l}{f}\arrow{e,t}{\Delta}
  \node{TV\otimes TV}\arrow{s,r} {(id\otimes f) + (f\otimes id)} \\
\node{T^{W}V}\arrow{s,l}{F}\arrow{e,t}{\Delta^{W}}
  \node{(TV\otimes T^{W}V)\oplus (T^{W}V\otimes TV)}\arrow{s,r}
  {(id\otimes F) \oplus (F\otimes id)} \\
\node{T^{Z}V}\arrow{e,b}{\Delta^{Z}} \node{(TV\otimes
  T^{Z}V)\oplus (T^{Z}V\otimes TV)}
\end{diagram}
\]
Therefore, the lower diagram has to commute if applied to any
element in $Im(f)\subset T^{W}V$. By Lemma \ref{3.3} there are
enough coderivations to imply the claim.
\end{proof}
Again let us describe $F$ by a system of maps.
\begin{lem}\label{4.2}
\begin{itemize}
\item [(a)]
Let $V$, $W$ and $Z$ be modules, and let  $\varrho:V^{\otimes
k}\otimes W\otimes V^{\otimes l}\to Z$ be a map, which may be viewed
as a map $\varrho: T^{W}V \to Z$ whose only nonzero component is the
original $\varrho$ on $V^{\otimes k}\otimes W\otimes V^{\otimes l}$.
Then $\varrho$ lifts uniquely to a map $\tilde{\varrho} :T^{W}V \to
T^{Z}V$
\[
\begin{diagram}
\node{} \node{T^{Z}V}\arrow{s,r}{projection} \\
\node{T^{W}V} \arrow{e,b}{\varrho} \arrow{ne,t}{\tilde{\varrho}}
\node{Z}
\end{diagram}
\]
which makes the diagram \eqref{eqn-4.1} in Proposition \ref{4.1}
commute. $\tilde{\varrho}$ is given by
$$ \tilde{\varrho}(v_{1},...,v_{r},w,v_{r+1},...,v_{r+s}):=0,
   \quad \text{ for } r<k \text{ or } s<l, $$
\begin{multline*}
\quad\quad\quad \tilde{\varrho}(v_{1},...,v_{r},w,v_{r+1},...,v_{r+s})\\
   := (-1)^{|\varrho|\sum_{i=1}^{r-k}|v_{i}|}
   (v_{1},...,\varrho(v_{r-k+1},...,w,...,v_{r+l}),...,v_{r+s}),\\
 \text{for } r\geq k \text{ and } s\geq l.
\end{multline*}
Thus $\tilde{\varrho}\mid_{V^{\otimes r}\otimes W\otimes V^{\otimes
s}}:V^{\otimes r}\otimes W\otimes V^{\otimes s} \to V^{\otimes r-k}
\otimes Z\otimes V^{\otimes s-l}$.
\item [(b)]
There is a one-to-one correspondence between maps $\sigma:T^{W}V\to
T^{Z}V$ making diagram \eqref{eqn-4.1} commute and systems of maps
$\{\varrho_{k,l}:V^{\otimes k}\otimes W\otimes V^{\otimes l}\to
Z\}_{k\geq 0, l\geq 0}$, given by $\sigma=\sum_{k\geq 0, l\geq 0}
\widetilde{\varrho_{k,l}}$.
\end{itemize}
\end{lem}
\begin{proof}
\begin{itemize}
\item [(a)]
Denote by $\tilde{\varrho}^{j}$ the component of $\tilde{\varrho}$
mapping $T^{W}V\to \bigoplus _{r+s=j} V^{\otimes r}\otimes Z\otimes
V^{\otimes s}$, and $\tilde{\varrho}^{-1}$ the component $T^W V\to R$. Then, $\tilde{\varrho}^{0},...,\tilde{\varrho}^
{m-1}$ uniquely determine the component $\tilde{\varrho}^{m}$.
$$ \Delta^{Z}(\tilde{\varrho}(v_{1},...,v_{r},w,v_{r+1},...,
   v_{r+s}))\quad\quad\quad\quad\quad\quad\quad\quad\quad\quad\quad\quad\quad $$
\begin{eqnarray*}
&=&(id\otimes \tilde{\varrho}+\tilde{\varrho} \otimes id)
   (\Delta^{W} (v_{1},...,v_{r},w,v_{r+1},...,v_{r+s}))\\
&=& \sum_{i=0}^{r}(-1)^{|\tilde{\varrho}|\sum_{t=1}^{i}|v_{t}|}
    (v_{1},...,v_{i}) \otimes \tilde{\varrho}(v_{i+1},...,w,...,v_{r+s})\\
& & +\sum_{i=r}^{r+s}\tilde{\varrho}(v_{1},...,w,...,v_{i})
    \otimes (v_{i+1},...,v_{r+s}).\\
\end{eqnarray*}
Projecting both sides to $R\otimes TV$ shows that $\tilde{\varrho}^{-1}=0$, and projecting for fixed $i+j=m$, to the component
\begin{multline*}
\quad\quad\quad V^{\otimes j}\otimes \bigoplus_{r+s=i} (V^{\otimes r}\otimes Z\otimes V^{\otimes s})+\bigoplus_{r+s=i} (V^{\otimes r}\otimes Z\otimes V^{\otimes s})\otimes V^{\otimes j}\\
\subset T^{Z}V\otimes TV+TV\otimes T^{Z}V,
\end{multline*}
yields for $ \Delta^{Z}(\tilde{\varrho}^{m}(v_{1},...,w,...,v_{r+s}))|_{V^{\otimes j}\otimes \bigoplus_{r+s=i} (V^{\otimes r}\otimes Z\otimes V^{\otimes s})+\bigoplus_{r+s=i} (V^{\otimes r}\otimes Z\otimes V^{\otimes s})\otimes V^{\otimes j}}$ the expression
\begin{equation*}
\pm (v_{1},...,v_{j}) \otimes
    \tilde{\varrho}^{i}(v_{j+1},...,w,...,v_{r+s})
+\tilde{\varrho}^{i}
    (v_{1},...,w,...,v_{r+s-j}) \otimes (v_{r+s-j+1},...,v_{r+s}).
\end{equation*}
For $m\geq 1$, choosing $i=m-1,\,\,\, j=1$ uniquely determines $\tilde{\varrho}^{m}$ by lower components. Therefore, an induction shows that
$\tilde{\varrho}^{m}$ is only nonzero on $V^{\otimes r}\otimes
W\otimes V^{\otimes s}$ with $r-k+s-l=m$, where
$\tilde{\varrho}^{m}(v_{1},...,v_{r},w,v_{r+1},...,v_{r+s})$ is
given by $ (-1)^{|\varrho|\sum_{i=1}^{r-k}|v_{i}|}
   (v_{1},...,\varrho(v_{r-k+1},...,w,...,v_{r+l}),...,v_{r+s})$.
\item [(b)]
Let $X:=\{\sigma:T^{W}V\to T^{Z}V\,\,|\,\, \sigma$ makes diagram
\eqref{eqn-4.1} commute$\}$. Then
\begin{eqnarray*}
& & \alpha:\{\{\varrho_{k,l}:V^{\otimes k}\otimes W\otimes
V^{\otimes l}\to Z\}_{k\geq 0, l\geq 0}\}
\to X, \\
& & \,\,\,\,\,\,\,\,\,\,\,\,\,\,\,\,\,\,\,
\{\varrho_{k,l}:V^{\otimes k}\otimes W\otimes V^{\otimes l}\to
Z\}_{k\geq 0, l\geq 0} \mapsto
\sum_{k\geq 0, l\geq 0} \widetilde{\varrho_{k,l}},\\
& & \beta:X \to \{\{\varrho_{k,l}:V^{\otimes k}\otimes
W\otimes V^{\otimes l}\to Z\}_{k\geq 0, l\geq 0}\},\\
& & \,\,\,\,\,\,\,\,\,\,\,\,\,\,\,\,\,\,\,
\sigma\mapsto\{pr_{Z}\circ\sigma|_{V^{\otimes k}\otimes W\otimes
V^{\otimes l}}\}_{k\geq 0, l\geq 0}
\end{eqnarray*}
are inverse to each other by (a).
\end{itemize}
\end{proof}
We may apply this to the Hochschild-complex.
\begin{defn}\label{4.3}
Let $(M,D^{M})$ and $(N,D^{N})$ be two A$_\infty$-bimodules over the
A$_\infty$-algebra $(A,D)$. Then a map $F: B^{M}A\to B^{N}A$ of
degree 0 is called an \textbf{A$_\infty$-bimodule-map}, if $F$ makes
the diagram
\[
\begin{diagram}
\node{B^{M}A}\arrow{s,l}{F}\arrow{e,t}{\Delta^{M}}
  \node{(BA\otimes B^{M}A)\oplus (B^{M}A\otimes BA)}\arrow{s,r}
  {(id\otimes F) \oplus (F\otimes id)} \\
\node{B^{N}A}\arrow{e,b}{\Delta^{N}} \node{(BA\otimes
  B^{N}A)\oplus (B^{N}A\otimes BA)}
\end{diagram}
\]
commute, and in addition $F\circ D^{M}=D^{N}\circ F$.

By Proposition \ref{4.1}, every A$_\infty$-bimodule-map induces a
map $F^{\sharp}:f\mapsto F\circ f$ between the Hochschild-complexes,
which preserves the differentials, since $(F^{\sharp}\circ
\delta^{M})(f)=F^{\sharp}(D^{M}\circ f+(-1)^{|f|} f\circ D)= F\circ
D^{M}\circ f+(-1)^{|f|}F\circ f\circ D= D^{N} \circ F\circ
f+(-1)^{|f|} F\circ f\circ D=\delta^{N}(F\circ f)=(\delta^{N}\circ
F^{\sharp})(f)$.
\end{defn}

\begin{prop}\label{4.4}
Let $(A,D)$ be an A$_\infty$-algebra with system of maps
$\{m_{i}:A^{\otimes i} \to A\}_{i\geq 1}$ from Proposition \ref{2.4}
associated to $D$, where $m_{0}=0$. Let $(M,D^{M})$ and $(N,D^{N})$
be A$_\infty$-bimodules over $A$ with systems of maps $\{b_{k,l}
:A^{\otimes k}\otimes M \otimes A^{\otimes l} \to M\}_{k\geq 0,
l\geq 0}$ and $\{c_{k,l} :A^{\otimes k}\otimes N \otimes A^{\otimes
l} \to N\}_{k\geq 0, l\geq 0}$ from Proposition \ref{3.7} associated
to $D^{M}$ and $D^{N}$ respectively. Let $F:T^{M}A \to T^{N}A$ be an
A$_\infty$-bimodule-map between $M$ and $N$, and let $\{F_{k,l}
:sA^{\otimes k}\otimes sM \otimes sA^{\otimes l} \to sN\}_{k\geq 0,
l\geq 0}$ be a system of maps associated to $F$ by Lemma \ref{4.2}
(b). Rewrite the maps $F_{k,l}$ by $f_{k,l} :A^{\otimes k}\otimes M
\otimes A^{\otimes l} \to N$ by using the suspension map: $F_{k,l}=
s\circ f_{k,l}\circ (s^{-1})^{\otimes k+l+1}$. Then the condition
$F\circ D^{M}=D^{N}\circ F$ is equivalent to the following system of
equations:
$$ f_{0,0}(b_{0,0}(m))=c_{0,0} (f_{0,0}(m)), $$
\begin{multline*}
f_{0,0}(b_{0,1}(m,a))- f_{0,1}(b_{0,0}(m),a)- (-1)^{|m|}
f_{0,1}(m,m_{1}(a))\\=c_{0,0}(f_{0,1}(m,a))+ c_{0,1}(f_{0,0}(m),a),
\end{multline*}
\begin{multline*}
f_{0,0}(b_{1,0}(a,m))- f_{1,0}(m_{1}(a),m)- (-1)^{|a|}
f_{1,0}(a,b_{0,0}(m))\\=c_{0,0}(f_{1,0}(a,m))+
c_{1,0}(a,f_{0,0}(m)),
\end{multline*}
$$...$$
\begin{multline*}
\sum_{i=1}^{k} \sum_{j=1}^{k-i+1} (-1)^{\varepsilon}
 f_{k-i+1,l}(a_{1},...,m_{i}(a_{j},...,a_{i+j-1}),...,m,...,a_{k+l+1})  \\
+\sum_{j=1}^{k} \sum_{i=k-j+2}^{k+l-j+2} (-1)^{\varepsilon}
 f_{j,k+l-i-j+3} (a_{1},...,b_{k-j+1,i+j-k-2}(a_{j},...,m,...,a_{i+j-1}),...,a_{k+l+1})  \\
+\sum_{i=1}^{l} \sum_{j=k+2}^{k+l-i+2} (-1)^{\varepsilon}
 f_{k,l-i+1}(a_{1},...,m,...,m_{i}(a_{j},...,a_{i+j-1}),...,a_{k+l+1}) \\
=\sum_{j=1}^{k+1} \sum_{i=k-j+2}^{k+l-j+2} (-1)^{\varepsilon'}
 c_{j,k+l-i-j+3} (a_{1},...,f_{k-j+1,i+j-k-2}(a_{j},...,m,...,a_{i+j-1}),...,a_{k+l+1})
\end{multline*}
In order to simplify notation, it is assume that in $(a_{1},...
,a_{k+l+1})$ above, only the first $k$ and the last $l$ elements are
elements of $A$ and $a_{k+1}=m\in M$. Then the signs are given by
\begin{eqnarray*}
\varepsilon &=& i\cdot \sum_{r=1}^{j-1}|a_{r}|+ (j-1) \cdot(i+1)+
(k+l+1)-i, \\ \text{and } \quad \varepsilon' &=&(i+1)\cdot
(j+1+\sum_{r=1}^{j-1}|a_{r}|).
\end{eqnarray*}
$$...$$
\end{prop}
\begin{proof}
The formula follows immediately from the explicit lifting properties
in Lemma \ref{3.5} (a) and Lemma \ref{4.2} (a).
\end{proof}
\begin{expl}\label{4.5}
Examples \ref{2.5} and \ref{3.8} can be extended in the following
way. Let $(A,\partial, \mu)$ be a differential graded algebra with
the A$_\infty$-algebra-structure $m_{1}:=\partial$, $m_{2}:=\mu$ and
$m_{k}:=0$ for $k\geq 3$. Let $(M,\partial^{M},\lambda^{M},
\rho^{M})$ and $(N,\partial^{N},\lambda^{N},\rho^{N})$ be
differential graded bimodules over $A$, with the A$_\infty
$-bialgebra-structures given by $b_{0,0}:=\partial^{M}$,
$b_{1,0}:=\lambda^{M}$, $b_{0,1}:=\rho^{M}$ and $b_{k,l}:=0$ for
$k+l>1$, and $c_{0,0}:=\partial^{N}$, $c_{1,0}:=\lambda^{N}$,
$c_{0,1}:=\rho^{N}$ and $c_{k,l}:=0$ for $k+l>1$. Finally, let
$f:M\to N$ be a bialgebra map of degree 0. Then $f$ becomes a map of
A$_\infty $-bialgebras by taking $f_{0,0}:=f$ and $ f_{k,l}:=0$ for
$k+l>0$. The equations from Proposition \ref{4.4} are the defining
conditions of a differential bialgebra map from $M$ to $N$:
\begin{eqnarray*}
f\circ\partial^{M} (m) &=& \partial^{N}\circ f(m)\\
f(m.a) &=& f(m).a\\
f(a.m) &=& a.f(m)
\end{eqnarray*}
There are no higher equations.
\end{expl}

\section{$\infty$-inner-products on A$_\infty$-algebras}\label{5}

There are canonical A$_\infty$-bialgebra-structures on a given
A$_\infty$-algebra $A$ and on its dual space $A^*$. We will define
$\infty$-inner products as A$_\infty$-bialgebra-maps from $A$ to
$A^*$.
\begin{lem}\label{5.1}
Let $(A,D)$ be an A$_\infty$-algebra., and let $D$ be given by the
system of maps $\{m_{i}: A^{\otimes i} \to A\}_{i\geq 1}$ from
Proposition \ref{2.4}.
\begin{itemize}
\item [(a)]
There is a canonical A$_\infty$-bimodule-structure on $A$ given by
$b_{k,l}:A^{\otimes k}\otimes A \otimes A^{\otimes l}\to A,
b_{k,l}:=m_{k+l+1}$.
\item [(b)]
There is a canonical A$_\infty$-bimodule-structure on $A^{*}$ given
by $b_{k,l}:A^{\otimes k}\otimes A^{*} \otimes A^{\otimes l}\to
A^{*}$,
\begin{multline*}
\quad\quad\quad (b_{k,l}(a_{1},...,a_{k},a^{*},a_{k+1},...,a_{k+l}))(a) \\
:=\pm a^{*}(m_{k+l+1}(a_{k+1},...,a_{k+l},a,a_{1},...,a_{k})),
\end{multline*}
with the signs from Lemma \ref{3.9}.
\end{itemize}
\end{lem}
\begin{proof}
\begin{itemize}
\item [(a)]
The A$_\infty$-bialgebra extension described in Lemma \ref{3.5} (a)
becomes the extension by coderivation described in Lemma \ref{2.3}
(a). Equations of Proposition \ref{3.7} become the equations of
Proposition \ref{2.4} and the diagram \eqref{eqn-3.1} from
Proposition \ref{3.4} becomes the usual coderivation diagram for
$D$.
\item [(b)] This follows from (a) and Lemma \ref{3.9}.
\end{itemize}
\end{proof}
\begin{expl}\label{5.2}
For a differential algebra $(A,\partial, \mu)$, the above
A$_\infty$-bialgebra structure on $A$ is exactly the bialgebra
structure given by left- and right-multiplication, since
$b_{1,0}(a\otimes b)=m_{2}(a\otimes b)=a\cdot b$ and
$b_{0,1}(a\otimes b)=m_{2}(a\otimes b)=a\cdot b$, for $a,b\in A$.

Similarly the A$_\infty$-bialgebra structure on $A^{*}$ is given by
right- and left-multiplication in the arguments: $b_{1,0}(a\otimes
b^{*})(c)=b^{*}(m_{2}(c\otimes a))=b^{*}(c\cdot a)$ and
$b_{0,1}(a^{*}\otimes b)(c)=a^{*}(m_{2}(b\otimes c))=a^{*}(b\cdot
c)$, for $a,b,c\in A$, and $a^{*},b^{*}\in A^{*}$.
\end{expl}
\begin{defn}\label{5.3}
Let $(A,D)$ be an A$_\infty$-algebra. Then, we call any
A$_\infty$-bimodule-map $F$ from $A$ to $A^{*}$ an
\textbf{$\infty$-inner-product} on $A$.
\end{defn}

\begin{prop}\label{5.4}
Let $(A,D)$ be an A$_\infty$-algebra. Then, specifying an
$\infty$-inner product on $A$ is equivalent to specifying a system
of inner-products on $A$, $\{<.,.,...>_{k,l}:A^{\otimes k+l+2} \to
R\}_{k\geq 0, l\geq 0}$ which satisfy the following relations:
$$ \sum_{i=1}^{k+l+2}(-1)^{\sum_{j=1}^{i-1}|a_{j}|}
 <a_{1},...,\partial(a_{i}),...,a_{k+l+2}>_{k,l}=\sum_{i,j,n}
 \pm <a_{i},...,m_{j}(a_{n},...),...>_{r,s}, $$
where in the sum on the right side, there is exactly one
multiplication $m_{j}$ ($j\geq 2$) inside the inner-product
$<...>_{r,s}$ and this sum is taken over all i, j, n subject to the
following conditions:
\begin{itemize}
\item [(i)]
The cyclic order of the $(a_{1},...,a_{k+l+2})$ is preserved.
\item [(ii)]
$a_{k+l+2}$ is always in the last slot of $<...>_{r,s}$.
\item [(iii)]
$a_{k+l+2}$ could be inside some $m_{j}$. By (ii), this is only the
case, when the first argument in the inner product is $a_{i}\neq
a_{1}$, as for example in the expression $<a_{i},...,m_{j}
(a_{n},..., a_{k+l+2},a_{1},...,a_{i-1})>_{r,s}$ for $i> 1$.
\item [(iv)]
The special arguments $a_{k+1}$ and $a_{k+l+2}$ are never multiplied
by $m_{j}$ at the same time.
\item [(v)]
The numbers $r$ and $s$ are uniquely determined by the position of
the element $a_{k+1}$ in the inner-product $<...>_{r,s}$. More
precisely, $a_{k+1}$ is in $(r+1)$-th spot of $<...>_{r,s}$, and $s$
is determined by the inner product $<...>_{r,s}$ having $r+s+2$
arguments.
\end{itemize}
\end{prop}
A graphical representation of the above conditions is given in
Definition \ref{5.5} and Example \ref{exp-4.1} below.
\begin{proof}
We use the description from Proposition \ref{4.4} for
A$_\infty$-bimodule-maps. An A$_\infty$-bimodule-map from $A$ to
$A^{*}$ is given by maps $f_{k,l}:A^{\otimes k}\otimes A \otimes
A^{\otimes l}\to A^{*}$, for $k,\,\,\, l\geq 0$. These are
interpreted as maps $A^{\otimes k}\otimes A \otimes A^{\otimes
l}\otimes A\to R$, which we denoted by the inner-product-symbol
$<...>_{k,l}$:
$$ <a_{1},...,a_{k+l+1},a'>_{k,l}:=(-1)^{|a'|}
   (f_{k,l}(a_{1},...,a_{k+l+1}))(a') $$
The A$_\infty$-bimodule-map condition from Proposition \ref{4.4}
becomes
\begin{multline}\label{eqn-cond}
 \sum \pm f_{k,l}(...,m_{i}(...),...,a,...)+
   \sum \pm f_{k,l}(...,b_{i,j}(...,a,...),...)\\
 + \sum \pm f_{k,l}(...,a,...,m_{i}(...),...)=
   \sum \pm c_{i,j}(...,f_{k,l}(...,a,...),...).
\end{multline}
Here $a\in A$ is the $(k+1)$-th entry of an element in $A^{\otimes
k}\otimes A \otimes A^{\otimes l}$, so that it comes from the
\textit{A$_\infty$-bimodule} $A$, instead of the
\textit{A$_\infty$-algebra} $A$.

Now, by Lemma \ref{5.1} (a), $b_{i,j}=m_{i+j+1}$ is one of the
multiplications from the A$_\infty$-structure, and therefore the
left side of the equation is $f_{k,l}$ applied to all possible
multiplications $m_{i}$. As $f_{k,l}$ maps into $A^{*}$, we may
apply the left hand side of \eqref{eqn-cond} to an element $a'\in A$
and use the notation $<...>_{k,l}$:
\begin{equation}\label{eqn-5.1}
  \sum \pm (f_{k,l}(...,m_{i}(...),...))(a') =
    \sum \pm <...,m_{i}(...),...,a'>_{k,l}.
\end{equation}
Next, use Lemma \ref{5.1} (b) to rewrite the right hand side of
\eqref{eqn-cond}:
\begin{equation}\label{eqn-5.2}
 \sum \pm (c_{i,j}(a_{1},...,f_{k,l}(...,a,...),...,a_{k+l+1}))(a')
\quad\quad\quad\quad\quad
\end{equation}
\begin{eqnarray*}
 &=& \sum \pm (f_{k,l}(...,a,...))(m_{r}(...,a_{k+l+1},a',a_{1},...)) \\
 &=& \sum \pm <...,a,...,m_{r}(...,a_{k+l+1},a',a_{1},...)>_{k,l}
\end{eqnarray*}
Equations \eqref{eqn-5.1} and \eqref{eqn-5.2} show that we take a
sum over all possibilities of applying one multiplication to the
arguments of the inner-product subject to the conditions (i)-(iv).
This is the statement of the Proposition, after isolating the
$\partial$-terms on the left. For condition (v), notice that the
extensions of $D$ and $D^{A}$ from Lemma \ref{2.3} (a) and Lemma
\ref{3.5} (a) record the special entry $a$ in the
A$_\infty$-bimodule $A$. Thus, the A$_\infty$-bimodule element $a$
determines the index $k$, and $l$ is determined by the number of
arguments of $<...>_{k,l}$.

An explicit check shows the correctness of the signs.
\end{proof}
There is a diagrammatic way of picturing Proposition \ref{5.4}.
\begin{defn}\label{5.5}
Let $(A,D)$ be an A$_\infty$-algebra with $\infty$-inner-product
$\{<.,.,...>_{k,l}:A^{\otimes k+l+2} \to R\}_{k\geq 0, l\geq 0}$. To
the inner-product $<...>_{k,l}$, we associate the symbol
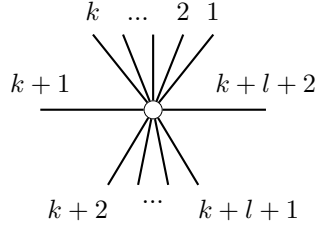
\[
\begin{pspicture}(0,0)(4,4)
 \psline(.5,2)(3.5,2)
 \psline(2,2)(1.2,3)
 \psline(2,2)(1.6,3)
 \psline(2,2)(2,3)
 \psline(2,2)(2.4,3)
 \psline(2,2)(2.8,3)
 \psline(2,2)(1.4,1)
 \psline(2,2)(1.8,1)
 \psline(2,2)(2.2,1)
 \psline(2,2)(2.6,1)
 \psdots[dotstyle=o,dotscale=2](2,2)
 \rput[b](2.8,3.2){$1$}  \rput[b](2.4,3.2){$2$}
 \rput[b](1.8,3.2){$...$} \rput[b](1.2,3.2){$k$}
 \rput[b](.5,2.2){$k+1$}
 \rput[b](3.5,2.2){$k+l+2$}
 \rput[tr](1.4,.8){$k+2$} \rput[t](2,.8){$...$} \rput[tl](2.6,.8){$k+l+1$}
\end{pspicture}
\]
More generally, to any inner-product which has (possibly iterated)
multiplications $m_{2}, m_{3}, m_{4}, ...$ (but without differential
$\partial=m_{1}$), such as
$$<a_{i},...,m_{j}(...),...,m_{p}(..., m_{q}(...) ,...),...>_{k,l},$$
we associate a diagram like above, by the following rules:
\begin{itemize}
\item[(i)]
To every multiplication $m_{j}$, associate a tree with $j$ inputs
and one output.
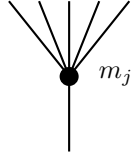
\[
\begin{pspicture}(0,0.5)(4,3.5)
 \psline(2,2)(1.2,3)
 \psline(2,2)(1.6,3)
 \psline(2,1)(2,3)
 \psline(2,2)(2.4,3)
 \psline(2,2)(2.8,3)
 \psdots[dotstyle=*,dotscale=2](2,2)
 \rput[l](2.4,2){$m_{j}$}
\end{pspicture}
\]
The symbol for the multiplication may also occur in a rotated way.
\item[(ii)]
To the inner product $<...>_{r,s}$, associate the open circle:
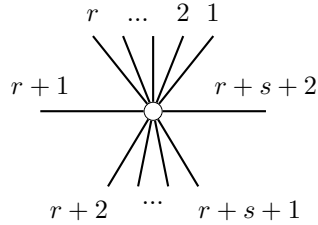
\[
\begin{pspicture}(0,0)(4,4)
 \psline(.5,2)(3.5,2)
 \psline(2,2)(1.2,3)
 \psline(2,2)(1.6,3)
 \psline(2,2)(2,3)
 \psline(2,2)(2.4,3)
 \psline(2,2)(2.8,3)
 \psline(2,2)(1.4,1)
 \psline(2,2)(1.8,1)
 \psline(2,2)(2.2,1)
 \psline(2,2)(2.6,1)
 \psdots[dotstyle=o,dotscale=2](2,2)
 \rput[b](2.8,3.2){$1$}  \rput[b](2.4,3.2){$2$}
 \rput[b](1.8,3.2){$...$} \rput[b](1.2,3.2){$r$}
 \rput[b](.5,2.2){$r+1$}
 \rput[b](3.5,2.2){$r+s+2$}
 \rput[tr](1.4,.8){$r+2$} \rput[t](2,.8){$...$} \rput[tl](2.6,.8){$r+s+1$}
\end{pspicture}
\]
There are $r$ elements attached at the top of the circle, and $s$
elements at the bottom of the circle, and the two (special) inputs
$(r+1)$ and $(r+s+2)$ are attached on the left and right. This gives
a total of $r+s+2$ inputs.
\item[(iii)]
The inputs $a_i$, for $i=1,\dots, r+s+2$, will be attached
counterclockwise, where the last element $a_{r+s+2}$ is in the far
right slot. For the multiplications $m_{j}$ of the graph, we use the
counterclockwise orientation of the plane to find the correct order
of the arguments $a_{i}$ in $m_{j}$.
\end{itemize}
We call these diagrams \textbf{inner-product-diagrams}.
\end{defn}

\begin{expl}\label{exp-4.1}
Let $a, b, c, d, e, f, g, h, i, j, k \in A$.
\begin{itemize}
\item
$<a,b,c,d>_{2,0}$, ($deg=2$):
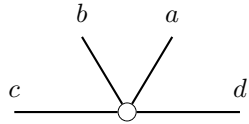
\[
\begin{pspicture}(0,0)(4,3)
 \psline(.5,1)(3.5,1)
 \psline(2,1)(1.4,2)
 \psline(2,1)(2.6,2)
 \psdots[dotstyle=o,dotscale=2](2,1)
 \rput[b](2.6,2.2){$a$}  \rput[b](1.4,2.2){$b$}
 \rput[b](0.5,1.2){$c$}  \rput[b](3.5,1.2){$d$}
\end{pspicture}
\]
\item
$<a,b,c,d,e,f,g,h,i>_{3,4}$, ($deg=7$):
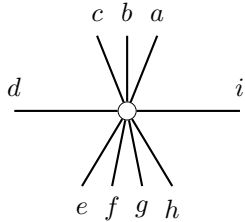
\[
\begin{pspicture}(0,0)(4,4)
 \psline(.5,2)(3.5,2)
 \psline(2,2)(1.6,3)
 \psline(2,2)(2,3)
 \psline(2,2)(2.4,3)
 \psline(2,2)(1.4,1)
 \psline(2,2)(1.8,1)
 \psline(2,2)(2.2,1)
 \psline(2,2)(2.6,1)
 \psdots[dotstyle=o,dotscale=2](2,2)
 \rput[b](2,3.2){$b$}  \rput[b](2.4,3.2){$a$}
 \rput[b](1.6,3.2){$c$}
 \rput[b](.5,2.2){$d$} \rput[b](3.5,2.2){$i$}
 \rput(1.4,.7){$e$} \rput(1.8,.7){$f$}
 \rput(2.2,.7){$g$} \rput(2.6,.7){$h$}
\end{pspicture}
\]
\item
$<m_{2}(m_{2}(b,c),m_{2}(d,e)),m_{2}(f,a)>_{0,0}$, ($deg=0$):
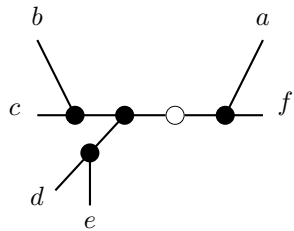
\[
\begin{pspicture}(0,0)(4,4)
 \psline(.5,2)(3.5,2)
 \psline(3,2)(3.5,3)
 \psline(1,2)(.5,3)
 \psline(1.66,2)(.74,1)
 \psline(1.2,1.5)(1.2,.8)
 \psdots[dotstyle=*,dotscale=2](1.66,2)
 \psdots[dotstyle=*,dotscale=2](1,2)
 \psdots[dotstyle=*,dotscale=2](1.2,1.5)
 \psdots[dotstyle=*,dotscale=2](3,2)
 \psdots[dotstyle=o,dotscale=2](2.33,2)
 \rput[b](.2,2){$c$}    \rput[b](3.8,2){$f$}
 \rput[b](3.5,3.2){$a$} \rput[b](0.5,3.2){$b$}
 \rput[b](.5,.8){$d$}   \rput[b](1.2,.5){$e$}
\end{pspicture}
\]
\item
$<a,b,m_{3}(c,d,m_{2}(e,f)),g,m_{2}(h,i))>_{1,2}$, ($deg=4$):
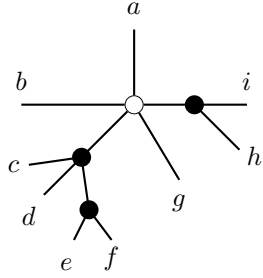
\[
\begin{pspicture}(0,-.8)(4,4)
 \psline(.5,2)(3.5,2)
 \psline(2,2)(2,3)
 \psline(2,2)(.8,.8)
 \psline(2,2)(2.6,1)
 \psline(2.8,2)(3.4,1.4)
 \psline(1.3,1.3)(.6,1.2)
 \psline(1.3,1.3)(1.4,.6)
 \psline(1.4,.6)(1.2,.2)
 \psline(1.4,.6)(1.7,.2)
 \psdots[dotstyle=o,dotscale=2](2,2)
 \psdots[dotstyle=*,dotscale=2](2.8,2)
 \psdots[dotstyle=*,dotscale=2](1.3,1.3)
 \psdots[dotstyle=*,dotscale=2](1.4,.6)
 \rput[b](2,3.2){$a$} \rput[b](.5,2.2){$b$}
 \rput(2.6,.7){$g$}   \rput[b](3.5,2.2){$i$}
 \rput[b](3.6,1.2){$h$}
 \rput[b](.4,1.1){$c$} \rput[b](.6,.4){$d$}
 \rput[b](1.1,-.2){$e$} \rput[b](1.7,-.2){$f$}
\end{pspicture}
\]
\item
$<c,m_{2}(d,e),m_{2}(m_{2}(f,g),h),i,m_{4}(j,k,a,b)>_{2,1}$,
($deg=5$):
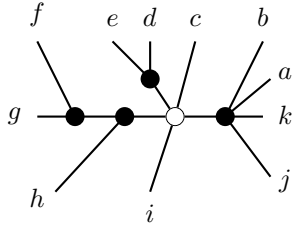
\[
\begin{pspicture}(0,0)(4,4)
 \psline(.5,2)(3.5,2)
 \psline(3,2)(3.5,3)
 \psline(1,2)(.5,3)
 \psline(1.66,2)(.74,1)
 \psline(3,2)(3.6,2.5)
 \psline(2.33,2)(2,1)
 \psline(3,2)(3.6,1.2)
 \psline(2.33,2)(2.6,3)
 \psline(2.33,2)(2,2.5)
 \psline(2,2.5)(2,3)
 \psline(2,2.5)(1.5,3)
 \psdots[dotstyle=*,dotscale=2](1.66,2)
 \psdots[dotstyle=*,dotscale=2](1,2)
 \psdots[dotstyle=*,dotscale=2](3,2)
 \psdots[dotstyle=*,dotscale=2](2,2.5)
 \psdots[dotstyle=o,dotscale=2](2.33,2)
 \rput[b](.2,1.9){$g$}  \rput[b](3.8,1.9){$k$}
 \rput[b](3.5,3.2){$b$} \rput[b](0.5,3.2){$f$}
 \rput[b](.5,.8){$h$}   \rput[b](3.8,2.5){$a$}
 \rput[b](2,.6){$i$}    \rput[b](3.8,1){$j$}
 \rput[b](2.6,3.2){$c$} \rput[b](2,3.2){$d$}
 \rput[b](1.5,3.2){$e$}
\end{pspicture}
\]
\end{itemize}
\end{expl}
\begin{defn}\label{5.6}
We define a chain-complex associated to inner-product-diagrams.

We define the degree of the inner-product-diagram associated to
$<...>_{k,l}$ with multiplications $m_{i_{1}},...,m_{i_{n}}$ to be
$k+l+\sum_{j=1}^{n} (i_{j}-2)$. Examples are given in \ref{exp-4.1}.
For $n\geq 0$, let $C_{n}$ be the space generated by
inner-product-diagrams of degree $n$. Then let $C:=\bigoplus_{n\geq
0} C_{n}$.

As for the differential $d$ on $C$, we use the composition with the
operator $\tilde{\partial}:=\sum_{i} id\otimes ...\otimes id\otimes
\partial \otimes id\otimes ... \otimes id$, where $\partial=m_{1}$
is at the $i$-th spot:
\begin{multline*}
 (d(<...,m(...,m(...),...),...>))(a_{1},...,a_{s})\\
 :=(<...,m(...,m(...),...),...> ) (\sum_{i=1}^{s} (-1)^{\sum
 _{j=1}^{i-1}|a_{j}|} (a_{1},...,\partial (a_{i}),...,a_{s}))
\end{multline*}

Some remarks and interpretations of this expression are in order.
First, consider the inner-product $<...>_{k,l}$ without any
multiplications. By Proposition \ref{5.4}, the differential applies
one multiplication into the inner-product-diagram in all possible
spots, such that the two lines on the far left and on the far right
are not being multiplied; compare Proposition \ref{5.4} (iv).

In the case, that multiplications are applied to the inner-product,
one can observe from Proposition \ref{2.3}, that $ \sum_{i}
m_{n}\circ(id\otimes ... \otimes \partial\otimes ... \otimes id) $
is given by the two terms
\begin{equation}\label{eqn-5.3}
\sum_{i} m_{n}\circ(id\otimes ... \otimes \partial\otimes ...
\otimes id)=\sum_{k=2}^{n-1} \sum_{i} m_{n+1-k} \circ(id\otimes ...
\otimes m_{k} \otimes ... \otimes id) + \partial \circ m_{n}.
\end{equation}
The sum over $i$ on both sides of the above equation applies $m_k$
to the $i$-th spot. The first term on the right hand side of
\eqref{eqn-5.3} transforms the multiplication $m_{n}$ into a sum of
all possible sompositions of $m_{n+1-k}$ and $m_{k}$:
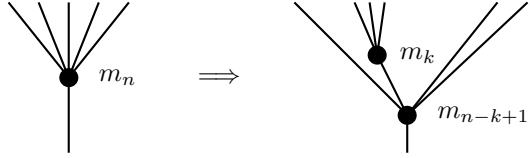
\[
\begin{pspicture}(0,0.5)(4.5,3.5)
 \psline(2,2)(1.2,3)
 \psline(2,2)(1.6,3)
 \psline(2,1)(2,3)
 \psline(2,2)(2.4,3)
 \psline(2,2)(2.8,3)
 \psdots[dotstyle=*,dotscale=2](2,2)
 \rput[l](2.4,2){$m_{n}$}
 \rput(4,2){$\Longrightarrow$}
\end{pspicture}
\begin{pspicture}(0,0.5)(4,3.5)
 \psline(2,1)(2,1.5)
 \psline(2,1.5)(.5,3)
 \psline(2,1.5)(3.2,3)
 \psline(2,1.5)(3.6,3)
 \psline(2,1.5)(1.6,2.3)
 \psline(1.6,2.3)(1.3,3)
 \psline(1.6,2.3)(1.5,3)
 \psline(1.6,2.3)(1.7,3)
 \psdots[dotstyle=*,dotscale=2](1.6,2.3)
 \rput[l](1.9,2.3){$m_{k}$}
 \psdots[dotstyle=*,dotscale=2](2,1.5)
 \rput[l](2.4,1.5){$m_{n-k+1}$}
\end{pspicture}
\]
The last term \eqref{eqn-5.3} is used for an inductive argument of
the above. One gets a term $\partial (m_{n}(...))$ attached to the
inner-product or possibly another multiplication, that has arguments
with $\tilde{\partial}$ applied, so that the above discussion can be
continued inductively.

We conclude, that the differential applies exactly one
multiplication in all possible spots, without multiplying the given
far left and far right inputs. Examples are given in Example
\ref{5.7} below.

It is $d:C_{n}\to C_{n-1}$, and $d^{2}=0$.
\begin{proof}
According to the definition of the degrees above, a multiplication
$m_{n}$ with $n$ inputs contributes by $n-2$. Taking the
differential applies one more multiplication in all possible ways.
If we attach $m_{n}$ to the diagram, then it replaces $n$ arguments
with one argument in the higher level. Therefore,
\begin{eqnarray*}
\text{new degree} & = & (\text{old degree})-n+1+(n-2) \\
                  & = & (\text{old degree})-1.
\end{eqnarray*}

We can prove $d^{2}=0$ in two ways:
\begin{itemize}
\item
Algebraically:

The definition of $d$ on the inner-products is given by composition
with the operator $\tilde{\partial}=\sum_{i} id\otimes ...\otimes
id\otimes \partial \otimes id\otimes ... \otimes id$, where
$\partial$ is in the $i$-th spot. Thus $d^{2}$ is composition with
$$ \tilde{\partial}^{2}=\sum_{i,j} \pm id\otimes ... \otimes \partial
\otimes ...\otimes \partial \otimes ... \otimes id=0. $$ This
vanishes, since the sum has two terms, where $\partial$ occurs at
the $i$-th and the $j$-th spot. This is obtained, by either first
applying $\partial$ to the $i$-th and then to the $j$-th spot, or
vice versa. These two possibilities cancel as $\partial$ is of
degree $-1$ and the first $\partial$ either has to move over the
second $\partial$, by which an additional minus sign is introduced,
or not.
\item
Diagrammatically (without signs):

$d$ applies one new multiplication to the inner-product-diagram, so
that $d^{2}$ applies two new multiplications. For two
multiplications, we have the following two possibilities.
\begin{itemize}
\item[(i)]
In the first case, the multiplications are on different outputs.
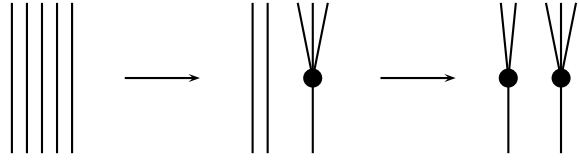
\[
\begin{pspicture}(0,0)(8.5,3)
 \psline(.5,.5)(.5,2.5)
 \psline(.7,.5)(.7,2.5)
 \psline(.9,.5)(.9,2.5)
 \psline(1.1,.5)(1.1,2.5)
 \psline(1.3,.5)(1.3,2.5)
 \psline{->}(2,1.5)(3,1.5)
 \psline(3.7,.5)(3.7,2.5)
 \psline(3.9,.5)(3.9,2.5)
 \psline(4.5,1.5)(4.3,2.5)
 \psline(4.5,.5)(4.5,2.5)
 \psline(4.5,1.5)(4.7,2.5)
 \psdots[dotstyle=*,dotscale=2](4.5,1.5)
 \psline{->}(5.4,1.5)(6.4,1.5)
 \psline(7,2.5)(7.1,1.5)
 \psline(7.2,2.5)(7.1,1.5)
 \psline(7.1,1.5)(7.1,.5)
 \psline(7.8,1.5)(7.6,2.5)
 \psline(7.8,.5)(7.8,2.5)
 \psline(7.8,1.5)(8,2.5)
 \psdots[dotstyle=*,dotscale=2](7.8,1.5)
 \psdots[dotstyle=*,dotscale=2](7.1,1.5)
\end{pspicture}
\]
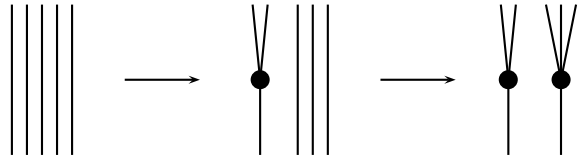
\[
\begin{pspicture}(0,0)(8.5,3)
 \psline(.5,.5)(.5,2.5)
 \psline(.7,.5)(.7,2.5)
 \psline(.9,.5)(.9,2.5)
 \psline(1.1,.5)(1.1,2.5)
 \psline(1.3,.5)(1.3,2.5)
 \psline{->}(2,1.5)(3,1.5)
 \psline(4.7,.5)(4.7,2.5)
 \psline(4.5,.5)(4.5,2.5)
 \psline(4.3,.5)(4.3,2.5)
 \psline(3.7,2.5)(3.8,1.5)
 \psline(3.9,2.5)(3.8,1.5)
 \psline(3.8,1.5)(3.8,.5)
 \psdots[dotstyle=*,dotscale=2](3.8,1.5)
 \psline{->}(5.4,1.5)(6.4,1.5)
 \psline(7,2.5)(7.1,1.5)
 \psline(7.2,2.5)(7.1,1.5)
 \psline(7.1,1.5)(7.1,.5)
 \psline(7.8,1.5)(7.6,2.5)
 \psline(7.8,.5)(7.8,2.5)
 \psline(7.8,1.5)(8,2.5)
 \psdots[dotstyle=*,dotscale=2](7.8,1.5)
 \psdots[dotstyle=*,dotscale=2](7.1,1.5)
\end{pspicture}
\]
The above figure shows that the final terms are obtained in two
different ways, which in fact cancel each other.
\item[(ii)]
The other possibility is to have multiplications on the same output.
Again, these terms may be obtained in two ways that cancel each
other:
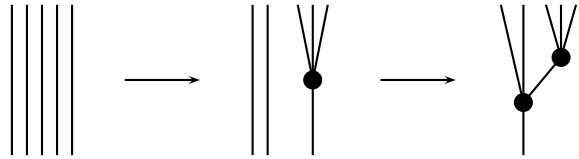
\[
\begin{pspicture}(0,0)(8.5,3)
 \psline(.5,.5)(.5,2.5)
 \psline(.7,.5)(.7,2.5)
 \psline(.9,.5)(.9,2.5)
 \psline(1.1,.5)(1.1,2.5)
 \psline(1.3,.5)(1.3,2.5)
 \psline{->}(2,1.5)(3,1.5)
 \psline(3.7,.5)(3.7,2.5)
 \psline(3.9,.5)(3.9,2.5)
 \psline(4.5,1.5)(4.3,2.5)
 \psline(4.5,.5)(4.5,2.5)
 \psline(4.5,1.5)(4.7,2.5)
 \psdots[dotstyle=*,dotscale=2](4.5,1.5)
 \psline{->}(5.4,1.5)(6.4,1.5)
 \psline(7,2.5)(7.3,1.2)
 \psline(7.3,2.5)(7.3,.5)
 \psline(7.3,1.2)(7.8,1.8)
 \psline(7.8,1.8)(7.6,2.5)
 \psline(7.8,1.8)(7.8,2.5)
 \psline(7.8,1.8)(8,2.5)
 \psdots[dotstyle=*,dotscale=2](7.8,1.8)
 \psdots[dotstyle=*,dotscale=2](7.3,1.2)
\end{pspicture}
\]
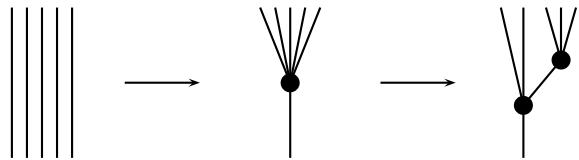
\[
\begin{pspicture}(0,0)(8.5,3)
 \psline(.5,.5)(.5,2.5)
 \psline(.7,.5)(.7,2.5)
 \psline(.9,.5)(.9,2.5)
 \psline(1.1,.5)(1.1,2.5)
 \psline(1.3,.5)(1.3,2.5)
 \psline{->}(2,1.5)(3,1.5)
 \psline(4.2,1.5)(3.8,2.5)
 \psline(4.2,1.5)(4.0,2.5)
 \psline(4.2,0.5)(4.2,2.5)
 \psline(4.2,1.5)(4.4,2.5)
 \psline(4.2,1.5)(4.6,2.5)
 \psdots[dotstyle=*,dotscale=2](4.2,1.5)
 \psline{->}(5.4,1.5)(6.4,1.5)
 \psline(7,2.5)(7.3,1.2)
 \psline(7.3,2.5)(7.3,.5)
 \psline(7.3,1.2)(7.8,1.8)
 \psline(7.8,1.8)(7.6,2.5)
 \psline(7.8,1.8)(7.8,2.5)
 \psline(7.8,1.8)(8,2.5)
 \psdots[dotstyle=*,dotscale=2](7.8,1.8)
 \psdots[dotstyle=*,dotscale=2](7.3,1.2)
\end{pspicture}
\]
\end{itemize}
\end{itemize}
\end{proof}
\end{defn}
\begin{expl}\label{5.7}
Let $a, b, c \in A$.
\begin{itemize}
\item
$k=0$, $l=0$: $d(<a,b>_{0,0})=0$
\[
\begin{pspicture}(0.6,0.4)(3.6,1.8)
 \psline(1.3,1)(2.7,1)
 \psdots[dotstyle=o,dotscale=2](2,1)
 \rput[b](1.3,1.1){$a$} \rput[b](2.8,1.1){$b$}
 \rput(.7,1.1){$\textrm{d}($} \rput(3.55,1.1){$)=0$}
\end{pspicture}
\]
\item
$k=1$, $l=0$: $d(<a,b,c>_{1,0})=<a\cdot b,c>_{0,0}\pm <b,c\cdot
a>_{0,0}$
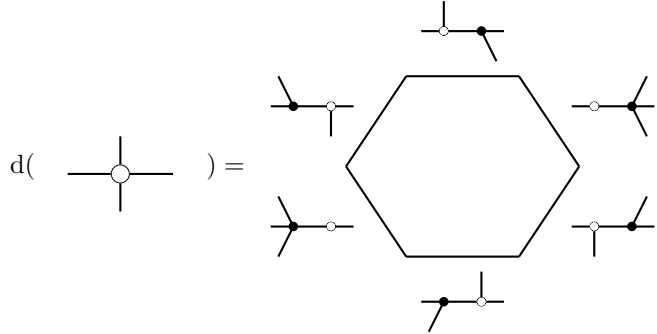
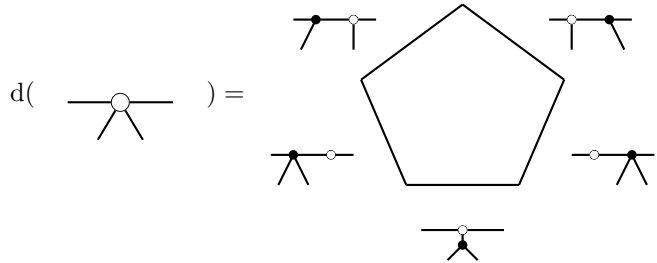
\[
\begin{pspicture}(0,0)(10,2.5)
 \psline(1.3,1)(2.7,1)
 \psline(2,1)(2,1.8)
 \psdots[dotstyle=o,dotscale=2](2,1)
 \rput[b](1.3,1.1){$b$} \rput[b](2.7,1.1){$c$} \rput[b](2.2,1.6){$a$}
 \rput(.7,1.1){$\textrm{d}($} \rput(3.4,1.1){$)=$}

 \psline(4.3,1)(5.7,1)
 \psline(4.8,1)(4.6,1.8)
 \psdots[dotstyle=*,dotscale=2](4.8,1)
 \psdots[dotstyle=o,dotscale=2](5.3,1)
 \rput[b](4.3,1.1){$b$} \rput[b](5.7,1.1){$c$} \rput[b](4.8,1.8){$a$}

 \rput(6.5,1.1){$\pm$}

 \psline(7.3,1)(8.8,1)
 \psline(8.3,1)(8.5,1.8)
 \psdots[dotstyle=*,dotscale=2](8.3,1)
 \psdots[dotstyle=o,dotscale=2](7.8,1)
 \rput[b](7.3,1.1){$b$} \rput[b](8.8,1.1){$c$} \rput[b](8.3,1.8){$a$}
\end{pspicture}
\]
\item
$k=0$, $l=1$: $d(<a,b,c>_{0,1})=<a\cdot b,c>_{0,0}\pm <a,b\cdot
c>_{0,0}$
\[
\begin{pspicture}(0,-.2)(10,1.8)
 \psline(1.3,1)(2.7,1)
 \psline(2,1)(2,0.2)
 \psdots[dotstyle=o,dotscale=2](2,1)
 \rput[b](1.3,1.1){$a$} \rput[b](2.7,1.1){$c$} \rput[b](2.2,0.3){$b$}
 \rput(.7,1.1){$\textrm{d}($} \rput(3.4,1.1){$)=$}

 \psline(4.3,1)(5.7,1)
 \psline(4.8,1)(4.6,0.2)
 \psdots[dotstyle=*,dotscale=2](4.8,1)
 \psdots[dotstyle=o,dotscale=2](5.3,1)
 \rput[b](4.3,1.1){$a$} \rput[b](5.7,1.1){$c$} \rput[b](4.8,0.2){$b$}

 \rput(6.5,1.1){$\pm$}

 \psline(7.3,1)(8.8,1)
 \psline(8.3,1)(8.5,0.2)
 \psdots[dotstyle=*,dotscale=2](8.3,1)
 \psdots[dotstyle=o,dotscale=2](7.8,1)
 \rput[b](7.3,1.1){$a$} \rput[b](8.8,1.1){$c$} \rput[b](8.3,0.2){$b$}
\end{pspicture}
\]
\end{itemize}
In the following three figures, where $k+l=2$, the righthand side is
understood to be a sum over the five, or respectively six,
inner-product-diagrams. Then, as $d^{2}=0$, the terms may be
arranged according to their boundaries. We obtain the polyhedra
associated to the inner-products $<...>_{k,l}$.
\begin{itemize}
\item
$k=2$, $l=0$:
\[
\begin{pspicture}(0,0)(10,6)
 \psline(1.3,2.9)(2.7,2.9) \psline(2,2.9)(2.3,3.4) \psline(2,2.9)(1.7,3.4)
 \psdots[dotstyle=o,dotscale=2](2,2.9)
 \rput(.7,3){$\textrm{d}($} \rput(3.4,3){$)=$}

 \psline(5.2,3.2)(6.55,4.2) \psline(6.55,4.2)(7.9,3.2)
 \psline(7.9,3.2)(7.3,1.8) \psline(7.3,1.8)(5.8,1.8)  \psline(5.8,1.8)(5.2,3.2)

 \psline(6,1.2)(7.1,1.2)
 \psline(6.55,1.4)(6.75,1.6) \psline(6.55,1.4)(6.35,1.6)
 \psdots[dotstyle=*,dotscale=1](6.55,1.4)
 \psline(6.55,1.2)(6.55,1.4) \psdots[dotstyle=o,dotscale=1](6.55,1.2)

 \psline(4,2.2)(5.1,2.2)
 \psline(4.3,2.2)(4.5,2.6) \psdots[dotstyle=*,dotscale=1](4.3,2.2)
 \psline(4.3,2.2)(4.1,2.6) \psdots[dotstyle=o,dotscale=1](4.8,2.2)

 \psline(8,2.2)(9.1,2.2)
 \psline(8.8,2.2)(8.6,2.6) \psdots[dotstyle=o,dotscale=1](8.3,2.2)
 \psline(8.8,2.2)(9,2.6) \psdots[dotstyle=*,dotscale=1](8.8,2.2)

 \psline(4.3,4)(5.4,4)
 \psline(4.6,4)(4.4,4.4) \psdots[dotstyle=*,dotscale=1](4.6,4)
 \psline(5.1,4)(5.1,4.4) \psdots[dotstyle=o,dotscale=1](5.1,4)

 \psline(7.7,4)(8.8,4)
 \psline(8,4)(8,4.4) \psdots[dotstyle=o,dotscale=1](8,4)
 \psline(8.5,4)(8.7,4.4) \psdots[dotstyle=*,dotscale=1](8.5,4)
\end{pspicture}
\]
\item
$k=1$, $l=1$:
\[
\begin{pspicture}(0,0)(10,6)
 \psline(1.3,2.9)(2.7,2.9) \psline(2,2.4)(2,3.4)
 \psdots[dotstyle=o,dotscale=2](2,2.9)
 \rput(.7,3){$\textrm{d}($} \rput(3.4,3){$)=$}

 \psline(5,3)(5.8,4.2) \psline(5.8,4.2)(7.3,4.2)  \psline(7.3,4.2)(8.1,3)
 \psline(8.1,3)(7.3,1.8) \psline(7.3,1.8)(5.8,1.8)  \psline(5.8,1.8)(5,3)

 \psline(6,1.2)(7.1,1.2)
 \psline(6.3,1.2)(6.1,.8)  \psdots[dotstyle=*,dotscale=1](6.3,1.2)
 \psline(6.8,1.2)(6.8,1.6) \psdots[dotstyle=o,dotscale=1](6.8,1.2)

 \psline(6,4.8)(7.1,4.8)
 \psline(6.3,4.8)(6.3,5.2) \psdots[dotstyle=o,dotscale=1](6.3,4.8)
 \psline(6.8,4.8)(7,4.4) \psdots[dotstyle=*,dotscale=1](6.8,4.8)

 \psline(4,2.2)(5.1,2.2)
 \psline(4.3,2.2)(4.1,2.6) \psdots[dotstyle=*,dotscale=1](4.3,2.2)
 \psline(4.3,2.2)(4.1,1.8) \psdots[dotstyle=o,dotscale=1](4.8,2.2)

 \psline(8,2.2)(9.1,2.2)
 \psline(8.3,2.2)(8.3,1.8) \psdots[dotstyle=o,dotscale=1](8.3,2.2)
 \psline(8.8,2.2)(9,2.6) \psdots[dotstyle=*,dotscale=1](8.8,2.2)

 \psline(4,3.8)(5.1,3.8)
 \psline(4.3,3.8)(4.1,4.2) \psdots[dotstyle=*,dotscale=1](4.3,3.8)
 \psline(4.8,3.8)(4.8,3.4) \psdots[dotstyle=o,dotscale=1](4.8,3.8)

 \psline(8,3.8)(9.1,3.8)
 \psline(8.8,3.8)(9,4.2) \psdots[dotstyle=o,dotscale=1](8.3,3.8)
 \psline(8.8,3.8)(9,3.4) \psdots[dotstyle=*,dotscale=1](8.8,3.8)
\end{pspicture}
\]
\item
$k=0$, $l=2$:
\[
\begin{pspicture}(0,0)(10,6)
 \psline(1.3,2.9)(2.7,2.9) \psline(2,2.9)(2.3,2.4) \psline(2,2.9)(1.7,2.4)
 \psdots[dotstyle=o,dotscale=2](2,2.9)
 \rput(.7,3){$\textrm{d}($} \rput(3.4,3){$)=$}

 \psline(5.2,3.2)(6.55,4.2) \psline(6.55,4.2)(7.9,3.2)
 \psline(7.9,3.2)(7.3,1.8) \psline(7.3,1.8)(5.8,1.8)  \psline(5.8,1.8)(5.2,3.2)

 \psline(6,1.2)(7.1,1.2)
 \psline(6.55,1)(6.75,.8) \psline(6.55,1)(6.35,.8)
 \psdots[dotstyle=*,dotscale=1](6.55,1)
 \psline(6.55,1.2)(6.55,1) \psdots[dotstyle=o,dotscale=1](6.55,1.2)

 \psline(4,2.2)(5.1,2.2)
 \psline(4.3,2.2)(4.5,1.8) \psdots[dotstyle=*,dotscale=1](4.3,2.2)
 \psline(4.3,2.2)(4.1,1.8) \psdots[dotstyle=o,dotscale=1](4.8,2.2)

 \psline(8,2.2)(9.1,2.2)
 \psline(8.8,2.2)(8.6,1.8) \psdots[dotstyle=o,dotscale=1](8.3,2.2)
 \psline(8.8,2.2)(9,1.8) \psdots[dotstyle=*,dotscale=1](8.8,2.2)

 \psline(4.3,4)(5.4,4)
 \psline(4.6,4)(4.4,3.6) \psdots[dotstyle=*,dotscale=1](4.6,4)
 \psline(5.1,4)(5.1,3.6) \psdots[dotstyle=o,dotscale=1](5.1,4)

 \psline(7.7,4)(8.8,4)
 \psline(8,4)(8,3.6) \psdots[dotstyle=o,dotscale=1](8,4)
 \psline(8.5,4)(8.7,3.6) \psdots[dotstyle=*,dotscale=1](8.5,4)
\end{pspicture}
\]
\end{itemize}
Finally, we graph the polyhedra in the case $k+l=3$. In general, the
polyhedron associated to $<...>_{k,l}$ is isomorphic to the one from
$<...>_{l,k}$. Furthermore, the polyhedra for $<...>_{n,0}$ and
$<...>_{0,n}$ are the ones known as Stasheff's associahedra.
\begin{itemize}
\item
The polyhedron for $k=3, l=0$ and for $k=0, l=3$:
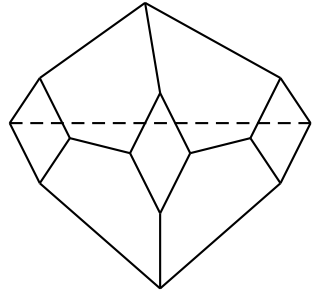
\[\begin{pspicture}(0,0.2)(4,4)
%
 \psline(.4,3)(0,2.4)    \psline(.4,1.6)(0,2.4)
 \psline(.8,2.2)(.4,1.6) \psline(.8,2.2)(.4,3)
 \psline(3.6,3)(4,2.4)     \psline(3.6,1.6)(4,2.4)
 \psline(3.2,2.2)(3.6,1.6) \psline(3.2,2.2)(3.6,3)
 \psline(2,2.8)(1.6,2)   \psline(2,2.8)(2.4,2)
 \psline(2,1.2)(1.6,2)   \psline(2,1.2)(2.4,2)
 \psline(.4,3)(1.8,4)   \psline(2,1.2)(2,0.2)
 \psline(3.6,3)(1.8,4)  \psline(3.6,1.6)(2,0.2)
 \psline(2,2.8)(1.8,4)  \psline(.4,1.6)(2,0.2)
 \psline(.8,2.2)(1.6,2) \psline(3.2,2.2)(2.4,2)
 \psline[linestyle=dashed](0,2.4)(4,2.4)
\end{pspicture}\]
\item
The polyhedron for $k=2, l=1$ and for $k=1, l=2$:
\[\begin{pspicture}(0,0)(3.4,4)
%
 \psline(2.2,4)(.4,3.6)   \psline(.4,3.6)(0,2.8)
 \psline(0,2.8)(1.2,2.6)  \psline(1.2,2.6)(2.6,3)
 \psline(2.6,3)(3,3.6)    \psline(3,3.6)(2.2,4)
 \psline(3,3.6)(3.4,2.8)  \psline(3.4,2.8)(3.3,1.6)
 \psline(3.3,1.6)(2.6,3)  \psline(1.2,2.6)(1.2,1.4)
 \psline(2.4,.8)(1.2,1.4) \psline(2.4,.8)(3.3,1.6)
 \psline(0,2.8)(0,1.8) \psline(.4,1)(0,1.8) \psline(.4,1)(1.2,1.4)
 \psline(1.6,.6)(.4,1)  \psline(1.6,.6)(2.4,.8)
 \psline[linestyle=dashed](0,1.8)(.4,2.4)
 \psline[linestyle=dashed](.4,3.6)(.4,2.4)
 \psline[linestyle=dashed](2.1,1.7)(1.4,2.2)
 \psline[linestyle=dashed](2,3)(1.4,2.2)
 \psline[linestyle=dashed](2.1,1.7)(2.8,2.4)
 \psline[linestyle=dashed](2,3)(2.8,2.4)
 \psline[linestyle=dashed](2.1,1.7)(2.4,.8)
 \psline[linestyle=dashed](1.4,2.2)(.4,2.4)
 \psline[linestyle=dashed](2.8,2.4)(3.4,2.8)
 \psline[linestyle=dashed](2,3)(2.2,4)
\end{pspicture}\]
\end{itemize}
\end{expl}

\end{document}